\documentclass[svgnames]{amsart}
\usepackage[english]{babel}
\usepackage[utf8]{inputenc}

\usepackage{hyperref}
\hypersetup{colorlinks,citecolor=Black,linkcolor=Black,urlcolor=Blue}

\usepackage{tabularx, ifthen}
\usepackage{amssymb, amsfonts, amsmath, amsthm, amscd, amsxtra, latexsym, mathabx, mathtools}
\usepackage{tikz}\usetikzlibrary{cd}
\usepackage{enumerate, enumitem}

\theoremstyle{plain}
\newtheorem{thm}{Theorem}[section] 
\newtheorem{lemma}[thm]{Lemma}
\newtheorem{cor}[thm]{Corollary}
\newtheorem{prop}[thm]{Proposition}

\theoremstyle{definition}
\newtheorem{defn}[thm]{Definition}
\newtheorem{notn}[thm]{Notation}
\newtheorem{rem}[thm]{Remark}
\newtheorem{ex}[thm]{Example}
\newtheorem{question}[thm]{Question}
\newtheorem{problem}[thm]{Problem}

\newcommand{\cal}[1]{\mathcal{#1}}

\newcommand{\calG}{\cal{MCG}}
\newcommand{\N}{\cal{N}}

\newcommand{\h}[1]{\widehat{#1}}
\newcommand{\Z}{\mathbb{Z}}
\newcommand{\R}{\mathbb{R}}
\newcommand{\frakS}{\mathfrak{S}}
\newcommand{\C}{\cal C}
\newcommand{\Bigdom}[1]{\operatorname{Big}\left(#1\right)}
\newcommand{\diam}{\mathrm{diam}}
\newcommand{\dist}{\mathrm{d}}
\newcommand{\nest}{\sqsubseteq}
\newcommand{\propnest}{\sqsubsetneq}
\newcommand{\orth}{\bot}
\newcommand{\trans}{\pitchfork}

\newcommand{\ofS}{\overline{\frakS}}
\newcommand{\wfS}{\widetilde{\frakS}}

\renewcommand{\ll }{\langle\hspace{-.7mm}\langle }
\newcommand{\rr }{\rangle\hspace{-.7mm}\rangle }

\usepackage{epigraph}

\title[Bounded cohomology, quotient extensions, and HHG]{Bounded cohomology, quotient extensions,\\and hierarchical hyperbolicity}

\author[F. Fournier-Facio]{Francesco Fournier-Facio}
    \address{(Francesco Fournier-Facio) Department of Pure Mathematics and Mathematical Statistics, University of Cambridge, Cambridge, UK}
    \email{ff373@cam.ac.uk}

\author[G. Mangioni]{Giorgio Mangioni}
    \address{(Giorgio Mangioni) Maxwell Institute and Department of Mathematics, Heriot-Watt University, Edinburgh, UK}
    \email{gm2070@hw.ac.uk}

\author[A. Sisto]{Alessandro Sisto}
    \address{(Alessandro Sisto) Maxwell Institute and Department of Mathematics, Heriot-Watt University, Edinburgh, UK}
    \email{a.sisto@hw.ac.uk}

\begin{document}
\begin{abstract}
    We call a central extension bounded if its Euler class is represented by a bounded cocycle. We prove that a bounded central extension of a hierarchically hyperbolic group (HHG) is still a HHG; conversely if a central extension is a HHG, then the extension is bounded, and under a further mild assumption the quotient is commensurable to a HHG. Motivated by questions on hierarchical hyperbolicity of quotients of mapping class groups, we therefore consider the general problem of determining when a quotient of a bounded central extension is still bounded, which we prove to be equivalent to an extendability problem for quasihomomorphisms. Finally, we show that quotients of the 4-strands braid group by suitable powers of a pseudo-Anosov are HHG, and in fact bounded central extensions of some HHG. We also speculate on how to extend the previous result to all mapping class groups.
\end{abstract}
\maketitle


\section{Introduction}
Central extensions $1\to K\to E\to G\to 1$ (which we often abbreviate with just $E$) are classified by their associated Euler class $\alpha_E\in H^2(G; K)$. We call a central extension \emph{bounded} if its Euler class is bounded, i.e. it is represented by a bounded cocycle. One reason of interest is the fact that bounded central extensions are quasi-isometrically trivial, meaning that $E$ is quasi-isometric to the direct product $K\times G$ \cite{Gersten_B_is_QIT}. Another reason of interest comes from the study of hierarchically hyperbolic groups, see below.

In this paper we always consider central extensions with finitely generated kernel. We are interested in the following natural problem:

\begin{problem}
\label{prob}
    Given a bounded central extension $E$, which of its quotient central extensions $\bar E$ are bounded?
\end{problem}

By a quotient central extension we mean that there exists a diagram as follows:
$$\begin{tikzcd}
    1\ar{r}&K \ar{r}\ar[d,"\cong"]&E\ar{r}\ar[d]&G\ar{r}\ar[d]&1.\\
    1\ar{r}& K\ar{r}&\bar E\ar{r}&\bar G\ar{r}&1.
\end{tikzcd}$$
Equivalently, $\bar E = E/N$, where $N$ is a normal subgroup of $E$ intersecting $K$ trivially.

We note in the other direction that if a central extension has a bounded quotient central extension, then it is bounded (Lemma \ref{qce pullback}). However, not all quotient extensions of a bounded central extension are themselves bounded. For instance, the Heisenberg group is a quotient central extension of $F_2\times \mathbb Z$ which is not bounded (see Example \ref{Heisenberg}); in fact, any central extension can be realised as a quotient of a trivial extension of a free group (see Example \ref{ex:all_ext_are_quotient}).

By Proposition \ref{bounded via qm:abelian_ker} (which was recently obtained independently by Tao and Wan \cite[Lemma 5.5]{tao2025properactionsfiniteproducts}), boundedness of a central extension $E$ is equivalent to the existence of a \emph{quasihomomorphism} $E\to K$ that is the identity on $K$. Quasihomomorphisms, in the sense of e.g. \cite{fujikapo}, are also relevant for quotient central extensions, and indeed we show that boundedness of a quotient central extension is equivalent to an extendability problem for a quasihomomorphism. An informal statement is as follows:

\begin{prop}[see Proposition \ref{bounded iff extendable}]
    \label{propintro:bounded-iff-ext}
    Consider a bounded central extension
    $$\begin{tikzcd}
    1\ar{r}& K \ar{r}&E\ar[r,"\pi"]&G\ar{r}&1.
\end{tikzcd}$$
with finitely generated kernel, and a quotient central extension $\bar E=E/N$. Then $\bar E$ is bounded if and only if a certain quasihomomorphism on $\pi(N)$ extends to $G$.
\end{prop}
This has further connection to bounded cohomology in degree 3, which we discuss in Section \ref{sec:background}. 

The second and third authors encountered Problem~\ref{prob} while studying quotients of mapping class groups, in particular relating to their conjecture on hierarchical hyperbolicity of quotients of mapping class groups \cite[Question 3]{MS_rigidity_MCG}. Hierarchically hyperbolic groups (HHG), as first defined by Behrstock, Hagen, and Sisto in \cite{HHS_I}, provide a common framework for, among others, mapping class groups of surfaces, most cubulated groups, fundamental groups of three-manifolds, and extra-large Artin groups \cite{hagensusse, HRSS_3manifold,ELTAG}, and are therefore amenable to tools from both low-dimensional topology and the world of CAT(0) cube complexes. Showing that a given group is hierarchically hyperbolic yields a lot of information about it (\cite{quasiflat,HHP:coarse,ANS:UEG,HHL:proximal,DMS:stable,FarrellJones} is a highly non-exhaustive list). Hence it is natural to explore which group-theoretic procedures the class of HHG is closed under. These include taking graph products \cite{graphprod}, relative hyperbolicity \cite{HHS_II}, many graphs of groups \cite{berlairobbio}, and several quotients \cite{BHMS, short_HHG:II,random_quot}. In this direction, we characterise which central extensions preserve hierarchical hyperbolicity:

\begin{thm}[see Theorem~\ref{thm:HHG_iff_bounded_abelian_ker}]
\label{thmintro:HHG_iff_bounded}
Let $1\to K\to E\to G\to 1$ be a central extension with finitely generated kernel, and suppose that $G$ is a hierarchically hyperbolic group. Then $E$ is a hierarchically hyperbolic group if and only if the extension is bounded.
\end{thm}
We postpone to Section~\ref{sec:quot_MCG} the details of how Theorem~\ref{thmintro:HHG_iff_bounded} relates to hierarchical hyperbolicity of quotients of mapping class groups. Actually, when we prove that a central extension with $E$ a HHG is bounded, we use no assumption on $G$:
\begin{thm}[see Theorem~\ref{thm:centralquot_is_HHS}]\label{thm:centralext_is_bdd:intro}
    Let $1\to K\to E\to G\to 1$ be a central extension with finitely generated kernel. If $E$ is a hierarchically hyperbolic group, then the extension is bounded. 
\end{thm}

Under the further mild assumption that $E$ has clean containers (see Definition~\ref{def:CC}), we prove that a central quotient of a HHG is commensurable to a HHG:

\begin{thm}[see Theorem~\ref{thm:centralquot_is_HHS}]\label{thm:centralquot_is_HHS:intro}
    Let $1\to K\to E\to G\to 1$ be a central extension with finitely generated kernel. If $E$ has a HHG structure with clean containers, there exists a finite-index subgroup $E'\le E$ containing $K$ such that $E'/K$ is a hierarchically hyperbolic group.
\end{thm}

\begin{rem}
    \cite[Question 6.8]{uniform_undistortion} asks if, given a group $G$ which admits an asymptotic cone with a cut-point, then every HHG $E$ which is quasi-isometric to $\Z\times G$ must contain a central element $t$ such that $E/\langle t\rangle$ is virtually a HHG. Under the assumption that $E$ has clean containers, our Theorem~\ref{thm:centralquot_is_HHS:intro} reduces the question to finding a central element.
\end{rem}

\begin{rem}
    With the notation of Theorem~\ref{thm:centralquot_is_HHS:intro}, $G$ is a finite-index overgroup of a HHG, and therefore inherits the structure of a hierarchically hyperbolic \emph{space} given by the quasi-isometric inclusion $E'/K\hookrightarrow G$. However, even assuming that $E$ has clean containers, one cannot hope that $G$ itself is a hierarchically hyperbolic \emph{group}: in Remark \ref{re-mark} we argue that the trivial central extension $E=\Z\times G$ of the $(3,3,3)$-triangle group $G$ is HHG, while $G$ is not by \cite{PS_unbounded}.
\end{rem}

\subsection*{Outline of sections and arguments}
We summarise here the techniques that are involved in our arguments. In what follows, we often restrict to central extensions with infinite cyclic kernel for the ease of exposition; in this setting, a quasihomomorphism with image in $\Z$ is simply called a \emph{quasimorphism}.

In Section \ref{sec:background} we discuss generalities on central extensions and their quotients, boundedness, and quasihomomorphisms. It is worth mentioning Lemma~\ref{qce pullback}, which states that, if a quotient extension is bounded, then so is the original extension. Next, we prove the characterisation of bounded quotient extensions, Proposition~\ref{bounded iff extendable}. The key observation is that a central extension $1 \to \Z \to E \xrightarrow{\pi} G \to 1$ is bounded if and only if it admits a quasimorphism $\phi\colon E\to \Z$ which is the identity on the kernel (see Proposition \ref{bounded via qm:abelian_ker}). The reader should think of $\phi$ as a ``coarse homomorphic'' retraction inducing a ``coarse'' splitting of the extension, as explained below.
\par\medskip
Section \ref{sec:HHG} starts with background material on hierarchically hyperbolic groups (HHG). Roughly, a group $\Gamma$ is a HHG if there exists a collection of hyperbolic spaces $\{\C U\}_{U\in \frakS}$, together with coarsely Lipschitz ``coordinate projections'' $\pi_U\colon \Gamma\to \C U$ satisfying various conditions. Each $U\in \frakS$ is called a \emph{domain}, and $\C  U$ is the associated \emph{coordinate space}.

We uncover a connection between boundedness and hierarchical hyperbolicity of a central extension. In one direction, Theorem~\ref{thm:HHG_iff_bounded_abelian_ker} shows that a bounded central extension of a HHG is itself a HHG. With an inductive argument, one can reduce to the case where the kernel is $\Z$, which was already settled by \cite[Theorem 5.14]{uniform_undistortion}; in turn, the proof of the latter is a refinement of \cite[Corollary 4.3]{HRSS_3manifold}, which solved the case where the base is hyperbolic. We sketch here the core idea of both arguments, as it enlightens several ideas that appear repeatedly throughout our paper. From a bounded extension one gets an unbounded quasimorphism $\phi\colon E\to \Z$ as above, and the function $E\to \Z\times G$ mapping $e\in E$ to $(\phi(e),\pi(e))$ is a quasi-isometry. In other words, a bounded extension is also \emph{quasi-isometrically trivial} in the sense of \cite{Gersten_B_is_QIT}, though the converse implication is not true in general (see \cite{frigsisto, ascarimilizia} and Remark \ref{rem:QITB}). Then the HHG structure for $E$ will have one domain for every domain of $G$, with projection to the same coordinate space factoring through the quotient map $\pi$, and one coordinate space quasi-isometric to $\R$ (a \emph{quasiline}) to ``detect'' the $\Z$-factor. Such quasiline is built out of the quasimorphism $\phi$, using an observation of Abbott, Balasubramanya, and Osin \cite[Lemma 4.15]{ABO}.
\par\medskip
In the opposite direction, Theorem~\ref{thm:centralquot_is_HHS} proves that, if $E$ is a HHG, then the central extension is bounded, moreover, under the further technical assumption of clean containers, the quotient is a finite-index overgroup of a HHG. To get boundedness we again look for a quasimorphism on $E$ which is the identity on $\Z$. Towards this, one first finds a coordinate space $\C U$ on which the centre acts loxodromically, which must be a quasiline; then the required quasimorphism is the \emph{Busemann quasimorphism} \cite[Section 4.1]{Manning_actions_on_hyp} associated to the action of (a finite-index subgroup of) $E$ on $\C U$, which roughly maps each $e\in E$ to its asymptotic translation length on the quasiline. 

We now sketch how to prove hierarchical hyperbolicity of a finite-index subgroup of $G$. Let $\C U_1,\ldots,\C U_k$ be the coordinate spaces on which $\Z$ acts loxodromically. By taking linear combinations of the associated Busemann quasimorphisms, and then replacing each $\C U_i$ with a new quasiline (again using \cite[Lemma 4.15]{ABO}), we can modify the HHG structure for $E$ in such a way that $\Z$ acts loxodromically on a single $\C U_i$, call it $\C U$, and with uniformly bounded orbits on every other coordinate space. Then the HHG structure of the quotient roughly coincides with what is left of the structure for $E$ after we delete $\C U$. 

In the above procedure, one has to restrict to a finite-index subgroup of $E$. More precisely, the $E$-action on $\frakS$ permutes the collection $\mathcal U=\{U_1,\ldots, U_k\}$, and the finite-index subgroup one considers is the kernel of the action on $\mathcal U$. Restricting to a subgroup is unavoidable, as in Remark~\ref{re-mark} we show that, while the $(3,3,3)$ triangle group $G$ is not a HHG by \cite{PS_unbounded}, the direct product $\Z\times G$ is. In this sense, the statement of Theorem~\ref{thm:centralquot_is_HHS:intro} is optimal.

\par\medskip
In Section \ref{sec:quot_MCG} we clarify how the study of hierarchical hyperbolicity of quotients of mapping class groups leads to questions on the boundedness of certain central extensions (Questions \ref{quest:mcg/dt_interior} and \ref{quest:quot_by_pA}), and discuss possible strategies to tackle them.

    The first version of this paper contained a proof of hierarchical hyperbolicity of quotients of the \emph{braid group} $B_4$ on four strands by a suitable power of a pseudo-Anosov (from now on, a \emph{pseudo-Anosov quotient}). This fitted our framework because the mapping class group of a surface with boundary is a central extension of (a finite-index subgroup of) the mapping class group of a surface without boundaries, whose pseudo-Anosov quotients were already studied in e.g. \cite{DGO,hhs_asdim}. The key idea was that, if a further quotient extension has hyperbolic base, then it is bounded by \cite{neumannreeves}, and therefore the original quotient extension is bounded by Lemma~\ref{qce pullback}. Thus the problem reduced to construct a suitable hyperbolic quotient of the base of the extension, which involved the machinery of \emph{rotating families} from \cite{dahmani:rotating} and of \emph{short HHG} from \cite{short_HHG:I}. This approach could possibly be adapted to higher-genus surfaces with boundary, conditionally to the residual finiteness of certain hyperbolic groups (see Remark~\ref{rem:enough_hyp_rf}).
    
    Since then, Tao proved that pseudo-Anosov quotients of \emph{any} mapping class group are hierarchically hyperbolic \cite[Corollary 1.7]{tao2025extensiontheoremquasimorphisms}. Tao's argument uses our Proposition~\ref{propintro:bounded-iff-ext} and Theorem~\ref{thmintro:HHG_iff_bounded} to reduce the problem to the extendability of a certain quasimorphism; then he proves a theorem which unifies most known extendability results for quasimorphisms. We still decided to include a streamline of our original arguments as a proof of concept, see Theorem~\ref{thm:braids}.

\subsection*{Acknowledgments} We would like to thank Federica Bertolotti, Roberto Frigerio, Shuhei Maruyama, Francesco Milizia, Davide Spriano, and Bingxue Tao for helpful discussions. We are also grateful to Mark Hagen for explaining to us how to cubulate suitable extensions of crystallographic groups, and to the anonymous referee for suggesting many improvements to the exposition. FFF is supported by the Herchel Smith Postdoctoral Fellowship Fund. GM is funded by an EPSRC-DTP PhD studentship.

\section{Bounded extensions and quasi(homo)morphisms}\label{sec:background}
In this section we discuss generalities on central extensions and bounded cohomology. In particular, we state a characterisation of boundedness in terms of quasimorphisms (Proposition \ref{bounded via qm:abelian_ker}). The main result in this section is Proposition \ref{bounded iff extendable}, which gives equivalent characterisations for a quotient central extension to be bounded.

\begin{notn}
    Throughout this section, we use $\sim$ to denote equality \emph{up to a uniformly bounded error}.
\end{notn}

\begin{defn}
    In this paper, \emph{cocycle} will always refer to \emph{inhomogeneous $2$-cocycle}, namely a map $\omega \colon G^2 \to K$ such that
    \[\delta\omega(g_1, g_2, g_3) = \omega(g_2, g_3) - \omega(g_1g_2, g_3) + \omega(g_1, g_2g_3) - \omega(g_1, g_2) = 0\]
    for all $g_1, g_2, g_3 \in G$.
\end{defn}

Let us briefly recall the dictionary between second cohomology and central extensions, referring the reader to \cite[Chapter IV.3]{brown} for details. Given a central extension $1\to K\to E\to G\to 1$, we pick a section $\sigma \colon G \to E$ that is \emph{normalised}, meaning that $\sigma(1) = 1$. We identify $\omega(g, h) = \sigma(g)\sigma(h)\sigma(gh)^{-1}$ with an element of $K$, and then $\omega \colon G^2 \to K$ is a cocycle that is \emph{normalised}, meaning that $\omega(1, g) = \omega(g, 1) = 0$ for all $g \in G$. This defines a class $[\omega] \in H^2(G; K)$, independent of the choice of $\sigma$, which is called the \emph{Euler class} of the extension and we also denote by $[E]$. Conversely, given a normalised cocycle $\omega \colon G^2 \to K$ we define a group $E$ with underlying set $K \times G$ and product
\[(k_1, g_1) \cdot (k_2, g_2) = (k_1 + k_2 + \omega(g_1, g_2), g_1 g_2).\]
Then $E$ is a central extension as above, and the map $g \to (0, g)$ is a normalised section. Equality in cohomology corresponds to equivalence of central extensions.

\begin{defn}
\label{def:bddce}
    Let $(K, \| \cdot \|)$ be an Abelian group endowed with a norm. For another group $G$, a class $\alpha \in H^2(G; K)$ is \emph{bounded} if it belongs to the image of the comparison map $H^2_b(G; K) \to H^2(G; K)$. More explicitly, $\alpha$ is bounded if there exists a cocycle $\omega : G^2 \to K$ such that $\alpha = [\omega]$, and such that $\|\omega(\cdot,\cdot)\|\colon G^2\to \Z$ is uniformly bounded.
    A central extension $1 \to K \to E \to G \to 1$ is \emph{bounded} if the corresponding Euler class $[E] \in H^2(G; K)$ is bounded.
\end{defn}

\begin{defn}
\label{def:qm}
    Let $(K, \| \cdot \|)$ be as above, and let $E$ be another group. A map $\chi\colon E\to K$ is a \emph{quasihomomorphism} if there exists $D(\chi)\ge 0$, called the \emph{defect} of $\chi$, such that, for every $e_1,e_2 \in E$, 
    $$\|\chi(e_1)+\chi(e_2)-\chi(e_1e_2)\|\le D(\chi).$$
    A quasihomomorphism is \emph{homogeneous} if it restricts to a homomorphism on every cyclic subgroup. When $K=\Z$ or $\R$ with the Euclidean norm, we say \emph{quasimorphism} instead of quasihomomorphism.
\end{defn}

\begin{rem}[Homogeneisation]\label{ref:plasmon}
    Given a quasimorphism $\phi\colon E\to \R$, for every $g\in E$ the limit $\phi_h(g)=\lim_{n\to \infty}\frac{\phi(g^n)}{n}$ always exists. The map $\phi_h\colon E\to \R$ is a homogeneous quasimorphism; moreover, both $\|\phi-\phi_h\|$ and the defect of $\phi_h$ are bounded in terms of the defect of $\phi$ \cite[Lemma 2.21]{calegari}.
\end{rem}

Real-valued quasimorphisms will only appear in the course of proofs about quasimorphisms with values in $\Z$. All of our statements will only involve the case in which $K$ is a finitely generated Abelian group endowed with a word norm. In this case, a subset of $K$ is bounded if and only if it is finite, so the notions above are independent of the choice of a norm. This allows to generalise the notion of quasihomomorphism to maps taking values in any discrete group, following Fujiwara--Kapovich \cite{fujikapo}.

\begin{defn}
\label{defn:qhm:noncommutative}
    A map $\chi\colon G\to H$ between groups is a \emph{quasihomomorphism} if there exist a finite set $F\subset H$ such that, for every $g_1,g_2\in G$, the difference between $\chi(g_1g_2)$ and $\chi(g_1)\chi(g_2)$ lies in $F$ (the way in which one takes the difference does not matter by \cite[Proposition 2.3]{heuer}).
\end{defn}

We record here an example of a quasimorphism that we shall use repeatedly throughout the paper.
\begin{ex}[{Busemann quasimorphism, see e.g. \cite[Section 4.1]{Manning_actions_on_hyp}}]\label{ex:busemann}
    Let $E$ be a group acting on a $\delta$-hyperbolic metric space $X$, with Gromov boundary $\partial X$, and suppose the action fixes an ideal point $p\in \partial X$. Given a sequence $\{x_n\}\in X$ converging to $p$, the map $\phi_{\{x\}}\colon E\to \R$ defined by
    $$\phi_{\{x\}}(g)=\limsup_{n\to \infty} \dist(gx_0,x_n)-\dist(x_0,x_n),$$
    is a quasimorphism of defect at most $16\delta$. Then the Busemann quasimorphism $\phi$ is the homogeneisation of $\phi_{\{x\}}$, as in Remark~\ref{ref:plasmon}. One can check that $\phi$ does not depend on the choice of the sequence; moreover, an element $g\in E$ acts loxodromically on $X$ if and only if $\phi(g)\neq 0$. Note that, by construction, for every $g\in E$ one has that $|\phi(g)|\le \dist(x_0,gx_0)$; conversely, if $X$ is a \emph{quasiline} (i.e. if $X$ is quasi-isometric to $\R$), then there exists a constant $L$ such that $\dist(x_0,gx_0)\le |\dist(gx_0,x_n)-\dist(x_0,x_n)|+L$, for all sufficiently large $n$ (if $\{x_0,gx_0,x_n\}=\{a,b,c\}$ and $b$ is ``between'' $a$ and $c$, one can see this by considering the distance between $b$ and a geodesic $[a,c]$). Taking the limsup, one gets that $\dist(x_0,gx_0)$ and $|\phi(g)|$ are within uniform distance.
\end{ex}

We now turn to the characterisation of when a central extension is bounded, in terms of the existence of certain quasihomomorphisms.

\begin{lemma}
\label{lem:normalisation}
    Let $1 \to K \to E \to G \to 1$ be a central extension.
    \begin{itemize}
        \item Every bounded cocycle $\omega \colon G^2 \to K$ is cohomologous to a normalised bounded cocycle.
        \item If there exists a quasihomomorphic section $\sigma \colon G \to E$, then there exists a normalised quasihomomorphic section.
    \end{itemize}
\end{lemma}

\begin{proof}
    For the first bullet, we have
    \[0 = \delta \omega(g^{-1}, g, 1) = \omega(g, 1) - \omega(1, 1) + \omega(g^{-1}, g) - \omega(g^{-1}, g).\]
    Therefore $\omega(g, 1) = \omega(1, 1)$, and similarly $\omega(1, g) = \omega(1, 1)$ for all $g \in G$. Let $b \colon G \to K$ be the constant function at $\omega(1, 1)$. Then $\delta b(g, h) = \omega(1, 1)$ for all $g, h \in G$. Setting $\omega' = \omega - \delta b$, we see that $\omega'$ is a normalised bounded cocycle cohomologous to $\omega$.
    
    For the second bullet, let $\sigma'(g) = \sigma(g)$ for all $g \neq 1$, and $\sigma'(1) = 1$. We need to check that $\sigma'(g_1) \sigma'(g_2)\sigma(g_1g_2)^{-1}$ takes finitely many values over $g_1, g_2 \in G$. If $g_1, g_2, g_1g_2 \neq 1$, then this follows from $\sigma$ being a quasihomomorphism. If one of $g_1, g_2$ is equal to $1$, then the equation above is just equal to $1$. Finally, if $g_1 = g \neq 1$ and $g_2 = g^{-1}$, then the equation above is equal to
    \[\sigma(g)\sigma(g^{-1}) = \big( \sigma(g)\sigma(g^{-1}) \sigma(1)^{-1} \big) \sigma(1),\]
    which again takes finitely many values because $\sigma$ is a quasihomomorphism.
\end{proof}

\begin{prop}\label{bounded via qm:abelian_ker}
    Let $1 \to K \to E \xrightarrow{\pi} G \to 1$ be a central extension with finitely generated kernel. The following are equivalent:
    \begin{enumerate}
        \item $[E]$ is bounded, i.e. $[E]$ is represented by a (normalised) bounded cocycle.
        \item There exists a quasihomomorphism $\chi \colon E \to K$ such that $\chi|_K$ is the identity.
        \item There exists a (normalised) section $s \colon G \to E$ which is a quasihomomorphism.
    \end{enumerate}
\end{prop}

Shortly after this paper was uploaded, Tao and Wan independently proved Proposition~\ref{bounded via qm:abelian_ker} with roughly the same methods (see \cite[Lemma 5.5]{tao2025properactionsfiniteproducts}).

\begin{proof}
    By Lemma \ref{lem:normalisation}, we may assume that all cocycles and sections involved are normalised.   
    Up to isomorphism of central extensions, $E$ has underlying set $K \times G$, and product
    \[(k_1, g_1) \cdot (k_2, g_2) = (k_1 + k_2 + \omega(g_1, g_2), g_1g_2);\]
    where $\omega \colon G^2 \to K$ is a normalised cocycle representing the central extension.

    (1)$\Rightarrow$(2): Suppose that the extension is bounded, and choose $\omega$ to be a bounded cocycle. We set $\psi \colon E \to K$ to be the projection onto the first factor; notice that $\psi|_K$ is the identity. Then
    \[\psi((k_1, g_1) \cdot (k_2, g_2)) = k_1 + k_2 + \omega(g_1, g_2) = \psi(k_1, g_1) + \psi (k_2, g_2) + \omega(g_1, g_2).\]
    Thus the defect of $\psi$ is bounded by the norm of $\omega$, and so it is bounded, i.e. $\psi$ is a quasihomomorphism.

    (2)$\Rightarrow$(1): Suppose that there exists a quasihomomorphism $\chi \colon E \to K$ such that $\chi|_K$ is the identity. Let $b \colon G \to K$ be defined as $b(g) = \chi(0, g)$. Then
    \begin{align*}
        (\omega - \delta b)(g_1, g_2) &= \omega(g_1, g_2) + b(g_1g_2) - b(g_1) - b(g_2) \\
        &= \chi(\omega(g_1, g_2), 1_G) + \chi(0, g_1g_2) - \chi(0, g_1) - \chi(0, g_2) \\
        &\sim \chi((\omega(g_1, g_2), 1_G)\cdot(0, g_1g_2)) - \chi(0, g_1) - \chi(0, g_2) \\
        &= \chi(\omega(g_1, g_2), g_1 g_2) - \chi(0, g_1) - \chi(0, g_2) \\
        &= \chi((0, g_1) \cdot (0, g_2)) - \chi(0, g_1) - \chi(0, g_2) \sim 0.
    \end{align*}
    Therefore $\omega - \delta b$ is bounded, and it is a cocycle cohomologous to $\omega$. This shows that the extension is bounded.

    (1)$\iff$(3): This equivalence was already pointed out in \cite{frigsisto} (see the discussion after Proposition 2.3 there). If $\sigma \colon G\to E$ is a quasihomomorphic normalised section, then the cocycle $\omega(g,h)=\sigma(gh)^{-1}\sigma(g)\sigma(h)$, which represents the Euler class, is bounded. The converse implication is \cite[Theorem C]{heuer}.
\end{proof}

\begin{cor}
\label{cor:decomposing:bounded:extensions}
    Let $1 \to K \to E \xrightarrow{\pi} G \to 1$ be a bounded central extension with finitely generated kernel. Let $L < K$ be a subgroup. Then the central extensions
    \[1 \to L \to E \to E/L \to 1\]
    and
    \[1 \to K/L \to E/L \to G \to 1\]
    are also bounded.
\end{cor}

\begin{proof}
    We use the characterisation in the second and third items of Proposition \ref{bounded via qm:abelian_ker}. Let $\chi \colon E \to K$ be a quasihomomorphism such that $\chi|_K$ is the identity. Since $K$ is a finitely generated Abelian group, we can find a quasihomomorphism $\xi\colon K\to L$ which is the identity on $L$. Then the composition $\xi \circ \chi\colon E\to L$ is a quasihomomorphism restricting to the identity on $L$, which proves that the $L$-central extension is bounded. 
    
    For the second extension, let $s\colon G\to E$ be a quasihomomorphic section. Composing $s$ with the projection $E \to E/L$, we get a quasihomomorphic section $G \to E/L$, which proves that the $E/L$-central extension is bounded.
\end{proof}

\subsection{(Bounded) quotient extensions}

Let $1 \to K \to E \xrightarrow{\pi} G \to 1$ be a central extension, and let $N \le E$ be a normal subgroup that intersects $K$ trivially, so that $\pi \colon N \to \pi(N)$ is an isomorphism. Then there is a diagram
\begin{equation}
\label{diagram qce}
\begin{tikzcd}
	1 & K & E & G & 1 \\
	1 & K & {\overline{E}} & {\overline{G}} & 1
	\arrow[from=1-1, to=1-2]
	\arrow[from=1-2, to=1-3]
	\arrow["{=}"', from=1-2, to=2-2]
	\arrow["\pi", from=1-3, to=1-4]
	\arrow["p"', from=1-3, to=2-3]
	\arrow[from=1-4, to=1-5]
	\arrow["p"', from=1-4, to=2-4]
	\arrow[from=2-1, to=2-2]
	\arrow[from=2-2, to=2-3]
	\arrow[from=2-3, to=2-4]
	\arrow[from=2-4, to=2-5]
\end{tikzcd}
\end{equation}
where $\overline{E} = E/N$ and $\overline{G} = G/\pi(N)$, and both rows are central extension. We use the letter $p$ to denote both quotients $E \to \overline{E}$ and $G \to \overline{G}$ by an abuse of notation. We say that the bottom row is a \emph{quotient central extension} of the top row. The main focus of the paper is the following question (see Problem \ref{prob}):

\begin{question}
\label{question qce}
    Under which conditions is a quotient central extension bounded?
\end{question}

A first easy necessary condition is given by the following observation:

\begin{lemma}
\label{qce pullback}
    With the notation of Diagram \eqref{diagram qce}, the pullback $p^* \colon H^2(\overline{G}; K) \to H^2(G; K)$ sends $[\overline{E}]$ to $[E]$. In particular, if $[\overline{E}]$ is bounded, then $[E]$ is bounded.
\end{lemma}

\begin{proof}
    We choose a normalised section $\overline \sigma \colon \overline G \to \overline E$. We also choose an injective map $\tau \colon \overline G \to E$ such that $p \tau = \overline \sigma$ and $\tau(1_{\overline{G}}) = 1_E$. This way Diagram \eqref{diagram qce} is enriched as follows:
\[\begin{tikzcd}
	1 & K & E & G & 1 \\
	1 & K & {\overline{E}} & {\overline{G}} & 1
	\arrow[from=1-1, to=1-2]
	\arrow[from=1-2, to=1-3]
	\arrow["{=}"', from=1-2, to=2-2]
	\arrow["\pi", from=1-3, to=1-4]
	\arrow["p"', from=1-3, to=2-3]
	\arrow[from=1-4, to=1-5]
	\arrow["p", from=1-4, to=2-4]
	\arrow[from=2-1, to=2-2]
	\arrow[from=2-2, to=2-3]
	\arrow["{\overline \pi}", shift left, from=2-3, to=2-4]
	\arrow["\tau"', from=2-4, to=1-3]
	\arrow["{\overline\sigma}", shift left, from=2-4, to=2-3]
	\arrow[from=2-4, to=2-5]
\end{tikzcd}\]
    Now $\pi \tau$ is a section for $p \colon G \to \overline G$, so every element of $G$ can be written uniquely as $g = \pi(n) \cdot \pi \tau(p(g))$ for some $n\in N$. We define $\sigma(g) = n \cdot \tau(p(g))$, so $\sigma \colon G \to E$.
    First, note that $\sigma$ is indeed a normalised section: $\sigma(1_G) = 1_E$, and if $g = \pi(n) \pi \tau (p(g))$, then
    \[\pi\sigma(g) = \pi(n \cdot \tau(p(g))) = \pi(n) \cdot \pi \tau(p(g)) = g.\]
    Secondly, we claim that $p \sigma = \overline \sigma p$, indeed
    \[p\sigma(g) = p(n \cdot \tau(p(g))) = p\tau(p(g)) = \overline \sigma(p(g)),\]
    which implies that the projection $p \colon E \to \overline E$ sends
    \[\sigma(g)\sigma(h)\sigma(gh)^{-1} \mapsto \overline \sigma(p(g)) \overline \sigma(p(h)) \overline \sigma(p(gh))^{-1}.\]
    This shows that the cocycle defined by $\sigma$ is indeed the pullback of the cocycle defined by $\overline \sigma$, and concludes the proof.
\end{proof}

This condition is however not sufficient, as the following examples show:

\begin{ex}
\label{Heisenberg}
    Consider the group $Z \times F_2$, where $Z = \langle z \rangle$ is infinite cyclic, and $F_2 = \langle x, y \rangle$ is free of rank $2$. Let $N$ be the normal closure of $z^{-1}[x, y]$, which intersects $Z$ trivially. The quotient $(Z \times F_2) / N$ is the Heisenberg group $H_3$, and the quotient $F_2 / \pi(N)$ is the free Abelian group $\mathbb{Z}^2$. So we have a map of central extensions:
\[\begin{tikzcd}
	1 & Z & {Z \times F_2} & {F_2} & 1 \\
	1 & Z & {H_3} & {\mathbb{Z}^2} & 1
	\arrow[from=1-1, to=1-2]
	\arrow[from=1-2, to=1-3]
	\arrow[from=1-2, to=2-2]
	\arrow[from=1-3, to=1-4]
	\arrow[from=1-3, to=2-3]
	\arrow[from=1-4, to=1-5]
	\arrow[from=1-4, to=2-4]
	\arrow[from=2-1, to=2-2]
	\arrow[from=2-2, to=2-3]
	\arrow[from=2-3, to=2-4]
	\arrow[from=2-4, to=2-5]
\end{tikzcd}\]
The top one splits, so it is trivial and in particular bounded. The bottom one is not trivial, and in fact it maps to a generator of $H^2(\mathbb{Z}^2; \mathbb{R}) \cong \mathbb{R}$ under the change of coefficients map. Because $\mathbb{Z}^2$ is amenable, $H^2_b(\mathbb{Z}^2; \mathbb{R})$ vanishes \cite[Chapter 3]{frigerio}, and so the bottom central extension cannot be bounded.
\end{ex}

\begin{ex}\label{ex:all_ext_are_quotient}
    More generally, we claim that every central extension can be expressed as a quotient of a trivial (thus bounded) central extension of a free group. Let $1 \to K \to E \to G \to 1$ be a central extension. Let $F$ be a free group and $p \colon F \to G$ a presentation, which we lift to a homomorphism $\tilde{p} \colon F \to E$, by choosing lifts in $E$ of the generators of $G$. Define $P \colon K \times F \to E$ to be the product of $\tilde{p}$ and the inclusion $K \to E$; this is a homomorphism because $K$ is central in $G$. An element $(k, w) \in K \times F$ belongs to the kernel of $P$ if and only if $\tilde{p}(w) = k$; in particular $K$ intersects this kernel trivially. So we indeed have a quotient central extension:
    \[\begin{tikzcd}
	1 & K & {K \times F} & {F} & 1 \\
	1 & K & {E} & {G} & 1
	\arrow[from=1-1, to=1-2]
	\arrow[from=1-2, to=1-3]
	\arrow[from=1-2, to=2-2]
	\arrow[from=1-3, to=1-4]
	\arrow[from=1-3, to=2-3]
	\arrow[from=1-4, to=1-5]
	\arrow[from=1-4, to=2-4]
	\arrow[from=2-1, to=2-2]
	\arrow[from=2-2, to=2-3]
	\arrow[from=2-3, to=2-4]
	\arrow[from=2-4, to=2-5]
\end{tikzcd}\]
\end{ex}

We can restate Question \ref{question qce} in terms of extendability of a certain quasihomomorphism. We start with an easy lemma, which is a variation of a well-known fact regarding real-valued homogeneous quasimorphisms.

\begin{lemma}
\label{induced qm}
    Let $N \le G$ be a normal subgroup, and denote by $p \colon G \to G/N$ the quotient. Let $\chi \colon G \to K$ be a quasihomomorphism. If $\chi|_N \equiv 0$, then there exists a quasihomomorphism $\phi \colon G/N \to K$ such that $\chi$ is at a bounded distance from $\phi p$. Moreover, if $H \le G$ is a subgroup such that $\chi|_H$ is a homomorphism, then $\phi$ can be chosen so that $\phi p|_H = \chi|_H$.
\end{lemma}

\begin{proof}
    Let $K = \Z^k \times T$, where $k \geq 0$ and $T$ is a finite abelian group. Accordingly, write $\chi = (\chi^1, \ldots, \chi^k, \chi^T)$, and note that $\chi^1, \ldots, \chi^k$ are quasimorphisms (while $\chi^T$ can be any function $G \to T$).    
    Recall that every quasimorphism $\chi^i \colon G \to \mathbb{R}$ is at a bounded distance from a unique homogeneous quasimorphism $\chi^i_h$ (Remark \ref{ref:plasmon}). The hypothesis, and uniqueness of homogeneous representatives, implies that $\chi^i_h|_N \equiv 0$ for all $1 \leq i \leq k$. Therefore, by \cite[Remark 2.90]{calegari} there exist homogeneous quasimorphisms $\phi^i_h \colon G/N \to \mathbb{R}$ such that $\chi^i_h = \phi^i_h p$ for all $1\le i\le k$. Setting $\phi^i$ to be the integral part of $\phi^i_h$, we obtain a quasimorphism $\phi^i \colon G/N \to \mathbb{Z}$ such that $\phi^i p \sim \phi^i_h p = \chi^i_h \sim \chi^i$. Setting $\phi = (\phi^1, \ldots, \phi^k, \phi^T)$, where $\phi^T \colon G/N \to T$ is any function, we obtain that $\chi$ is at a bounded distance from $\phi p$. Notice that $\phi$ is indeed a quasihomomorphism, as its entries are quasimorphisms. 

    Suppose that $\chi|_H$ is a homomorphism. Then by uniqueness of homogeneous representatives, $\chi^i|_H = \chi^i_h |_H = \phi^i_h p|_H$ for all $i$. In particular, $\phi^i_h$ takes integer values on $p(H)$, and so $\phi^i|_{p(H)} = \phi^i_h|_{p(H)}$. Moreover, because $\chi|_N \equiv 0$, there exists a map $\phi^T \colon G/N \to T$ such that $\phi^T p|_H = \chi^T|_H$. With this choice of $\phi^T$, we have $\phi p|_H = \chi|_H$, as promised.
\end{proof}

Back to the setting of Diagram \eqref{diagram qce}, we now provide the characterisation of boundedness of quotient central extensions in terms of extendable quasihomomorphisms. Recall from Lemma \ref{qce pullback} that, in order for the quotient central extension to be bounded, the original central extension needs to be bounded, so we assume this throughout.

\begin{prop}
\label{bounded iff extendable}
    With the notation of Diagram \eqref{diagram qce}, suppose that $[E]$ is bounded, and let $\chi \colon E \to K$ be a quasihomomorphism such that $\chi|_K$ is the identity, provided by Proposition \ref{bounded via qm:abelian_ker}. Then:
    \begin{itemize}
        \item The map $\chi \pi^{-1} \colon \pi(N) \to K$ is a well-defined quasihomomorphism, and it is almost invariant under conjugacy by $G$. 
        \item If $\omega$ is a normalised bounded cocycle representing $[E]$, then $\chi$ can be chosen so that $\delta(\chi\pi^{-1}) = \omega|_{\pi(N)}$.
        \item $[\overline{E}]$ is bounded if and only if $\chi \pi^{-1}$ can be extended to a quasihomomorphism $G \to K$, for some (equivalently, every) choice of $\chi$.
    \end{itemize}
\end{prop}

Given a normal subgroup $\Lambda \le \Gamma$, we say that a quasihomomorphism $\phi \colon \Lambda \to K$ is \emph{almost invariant under conjugacy by $\Gamma$} if there exists a constant $C > 0$ such that
\[\| \phi(\gamma \lambda \gamma^{-1}) - \phi(\lambda) \| < C\]
for all $\gamma \in \Gamma, \lambda \in \Lambda$. Note that quasihomomorphisms on $\Gamma$ are almost invariant under conjugacy by $\Gamma$, so this is a necessary condition for a quasihomomorphism on $\Lambda$ to be extendable, which will come up in the next result.

\begin{proof}[Proof of Proposition~\ref{bounded iff extendable}]
    $\chi \pi^{-1}$ is well-defined because $\pi|_N$ is an isomorphism. Because $\chi$ is defined on $E$, it is almost invariant under conjugacy by $E$, and then it follows that $\chi \pi^{-1}$ is almost invariant under conjugacy by $G$. 
    
    Moving to the second bullet, from the proof of Proposition \ref{bounded via qm:abelian_ker} we see that if $\omega$ is a normalised bounded cocycle representing $E$, seen as the Cartesian product $K \times G$ with the product twisted by $\omega$, then $\chi$ can be chosen to be the projection onto the first factor, while $\pi$ is the projection onto the second factor. Now every element $g \in N$ can be written uniquely as $(b(g), \pi(g))$, for some $b(g) \in K$. It follows that $\chi \pi^{-1} \colon \pi(N) \to K$ is just the map $\pi(g) \mapsto b(g)$. Because $N$ is a group and $\pi$ is injective, for all $g_1, g_2 \in N$ it holds
    \begin{align*}
        (b(g_1g_2), \pi(g_1g_2)) &= (b(g_1), \pi(g_1)) \cdot (b(g_2), \pi(g_2)) \\
        &= (b(g_1) + b(g_2) + \omega(\pi(g_1), \pi(g_2)), \pi(g_1g_2));
    \end{align*}
    which shows that $\delta (\chi \pi^{-1}) = \omega$ as cocycles $\pi(N)^2 \to K$.

    We finally prove the equivalence in the third bullet. Suppose that $\chi$ is such that $\chi \pi^{-1}$ extends to a quasihomomorphism $\phi \colon G \to K$. Consider the quasihomomorphism $\chi - \phi \pi \colon E \to K$. Notice that $\chi - \phi \pi$ restricts to the zero map on $N$ and to the identity on $K$. Therefore, Lemma \ref{induced qm} yields a quasihomomorphism $\psi \colon \overline{E} \to K$ such that $\psi p$ is at a bounded distance from $\chi - \phi \pi$, and restricts to the identity on $K$. Thus $[\overline{E}]$ is bounded, by the second item of Proposition~\ref{bounded via qm:abelian_ker}.

    Now suppose that $[\overline{E}]$ is bounded, so by Proposition~\ref{bounded via qm:abelian_ker} there exists a quasihomomorphism $\psi \colon \overline{E} \to K$ that restricts to the identity on $K$. Then $\psi p \colon E \to K$ vanishes on $N$ and is still the identity on $K$. Therefore $\chi - \psi p \colon E \to K$ vanishes on $K$, and coincides with $\chi$ on $N$, for any choice of $\chi$. By Lemma \ref{induced qm}, there exists a quasihomomorphism $\phi \colon G \to K$ such that $\phi \pi \colon E \to K$ is at a bounded distance from $\chi - \psi p$. So, up to changing $\phi$ by a bounded amount on $\pi(N)$, it is an extension of $\chi \pi^{-1}$.
\end{proof}

\begin{rem}
    Shuhei Maruyama pointed out to us that both Propositions~\ref{bounded via qm:abelian_ker} and~\ref{bounded iff extendable} can also be proved by diagram chasing, using the exact sequences from \cite{extendable:main}.
\end{rem}

For completeness, let us include an alternative characterisation of boundedness of quotient central extensions, in terms of sections that are quasihomomorphisms, as in the last item of Proposition~\ref{bounded via qm:abelian_ker}.

\begin{prop}
\label{bounded iff compatible section}
    With the notation of Diagram \eqref{diagram qce}, suppose that $[E]$ is bounded, so that there exists a section $s \colon G \to E$ of $\pi$ that is a quasihomomorphism, by Proposition~\ref{bounded via qm:abelian_ker}. Then $[\overline{E}]$ is bounded if and only $s$ can be chosen so that $s(\pi(N)) \subset N$.
\end{prop}

\begin{proof}
    Suppose first that $[\overline{E}]$ is bounded, so by Proposition~\ref{bounded via qm:abelian_ker} there exists a quasihomomorphism $\chi \colon \overline{E} \to K$ such that $\chi|_K$ is the identity. Since $[E]$ is bounded, by Proposition~\ref{bounded via qm:abelian_ker} again, there exists a section $s \colon G \to E$ that is a quasihomomorphism. We define $\hat{s}(g) = s(g) \cdot \chi(ps(g))^{-1}$.
    Since $\chi$ takes values in $K$, $\hat{s}$ is still a section of $\pi$. As a composition of (quasi)homomorphisms $g \mapsto (\chi p s)^{-1}$ is a quasihomomorphism. Moreover it is easy to see that the product of two quasihomomorphisms is a quasihomomorphism provided that their images commute (this is the same proof as the fact that a sum of quasimorphisms is a quasimorphism). Since $K$ is central in $E$, we conclude that $\hat{s}$ is a quasihomomorphism. Finally, for $n \in N$, the element $s\pi(n)$ can be written uniquely as $k m$, for $k \in K, m \in N$. Then $\chi(p(km)) = \chi(k) = k$, by the choice of $\chi$. Hence $\hat{s}(\pi(n)) = km \cdot k^{-1} = m \in N$.

    Conversely, suppose that there exists a section $s \colon G \to E$ of $\pi$ that is a quasihomomorphism, with the property that $s (\pi (N)) \subset N$, and let $F$ denote the finite set in Definition \ref{defn:qhm:noncommutative}. By Proposition \ref{bounded via qm:abelian_ker}, we need to find a section $\bar{s} \colon \overline{G} \to \overline{E}$ of $\bar{\pi} \colon \overline{E} \to \overline{G}$ that is a quasihomomorphism. For an element $g \in G$ we define $\bar{s} \colon \overline{G} \to \overline{E}$ as $\bar{s} p (g) = p s (g)$. This is \emph{almost} well-defined, indeed for $g \in G, n \in N$
    \[ps(g  \pi(n)) \in p(s(g) s\pi(n) \cdot F) = ps(g) \cdot p(F),\]
    where we used that $s\pi(n)\in N$ and therefore $ps\pi (n)=1$. So there is a finite amount of ambiguity in our definition of $\bar{s}$, and we can just choose one value among the possible ones. We have $\bar{\pi} \bar{s} p(g) = \bar{\pi} p s(g) = p \pi s(g) = p(g)$, so $\bar{s}$ is a section. To see that it is a quasihomomorphism, we compute (accounting for the ambiguity in the definition):
    \begin{align*}
        \bar{s}p(gh) &\in ps(gh) \cdot p(F) \subset p s(g) \cdot ps(h) \cdot p(F)^2 \\
        &\subset \bar{s} p(g) p(F) \cdot \bar{s}p(h) p(F) \cdot p(F)^2 \\
        &\subset \bar{s}p(g) \bar{s}p(h) \cdot p(F)^2 p(F)^{-1} p(F)^3;
    \end{align*}
    where the last inclusion used the fact that, for a quasihomomorphism $f \colon A \to B$, if $D$ is the set in Definition \ref{defn:qhm:noncommutative}, then for all $a \in A$ it holds $f(a)^{-1} D f(a) \subset D^2 D^{-1}$ \cite[Section 2.3]{fujikapo}. Since the set $p(F)^2 p(F)^{-1} p(F)^3$ is still finite, we conclude.
\end{proof}

\subsection{Extending quasimorphisms}
In view of Proposition~\ref{bounded iff extendable}, the problem of whether the quotient central extension $[\overline{E}]$ is bounded is equivalent to the problem of whether a specific quasihomomorphism $\pi(N) \to K$ extends to $G$. This problem is non-trivial. On the one hand, the easy obstruction for extendability, that is almost invariance under conjugacy by $G$, is excluded by the the first bullet of Proposition \ref{bounded iff extendable}. On the other hand, almost invariance under conjugacy is not sufficient for a quasihomomorphism to be extendable in general, as we have seen in Example \ref{Heisenberg}. 

\begin{rem}\label{rem:extendable_boundary}
    In the context of Proposition \ref{bounded iff extendable}, the quasihomomorphism one has to extend is always such that its boundary is extendable. This shows a stark difference between the problem of extending bounded cohomology classes and quasihomomorphisms \cite[end of Section 1.3]{extendable:survey}.
\end{rem}

\subsubsection{Extendability results in the literature}
The problem of extending quasihomomorphisms has been studied at length in recent years, and produced various examples of non-extendable quasimorphisms, and criteria for extendability, see \cite{extendable:survey} for a survey. In particular, we recall the following criteria:

\begin{prop}[{\cite[Proposition 1.6]{extendable:bavard}} and {\cite[Theorem 1.9]{extendable:main}}]\label{thm:bavard}
    Let $1 \to N \to G \to G/N \to 1$ be a short exact sequence of groups. Then any of the following conditions guarantees the extendability of every quasihomomorphism $N \to K$ that is almost invariant under conjugacy by $G$.
    \begin{enumerate}
        \item The extension virtually splits.
        \item $H^3_b(G/N; \mathbb{R}) = H^2_b(G/N; \mathbb{R}) = H^2(G/N; \mathbb{R}) = 0$.
    \end{enumerate}
\end{prop}

In another direction, very far from the case of a normal subgroup, we have:

\begin{prop}[\cite{Hull_Osin}, see also \cite{FPS}]
\label{extending hypemb}
    Let $H \le G$ be a hyperbolically embedded subgroup. Then every quasihomomorphism $H \to K$ extends to a quasihomomorphism $G \to K$.
\end{prop}

For both propositions, the original statement is about real-valued quasimorphisms, however these statements follow by splitting $K$ as $\Z^k \times T$, choosing arbitrary extensions for the $T$ part, and dealing with the $\Z$ factors by taking the integral part and modifying by a bounded amount.

Recently, Tao developed a very general framework that unifies the results for hyperbolically embedded subgroups and for normal subgroups, and has many new consequences \cite{tao2025extensiontheoremquasimorphisms}.

\subsubsection{Extending quasimorphisms from a union}
We conclude the section by proving a new extendability result, which applies to subgroups obtained as an increasing union of subgroups where quasimorphisms can be extended. We hope that this will find applications towards extending quasimorphisms from ``small-cancellation-like'' normal subgroups.

\begin{lemma}\label{lem:extension_from_union}
    Let $G$ be a finitely generated group, $N \le G$ a normal subgroup, and $\chi \colon N \to \mathbb{Z}$ a quasimorphism. Suppose that $N$ can be expressed as a directed union of subgroups $\{ N_i \}_{i \in I}$, and that for each $i \in I$ there is a quasimorphism $X_i \colon G \to \mathbb{Z}$ that extends $\chi|_{N_i}$, such that the defect of the $X_i$ is uniformly bounded. Then $\chi$ extends to a quasimorphism $X \colon G \to \mathbb{Z}$.
\end{lemma}

\begin{proof}
    By using homogeneous representatives for one direction, and integer parts for the other, as we did before, we can replace $\Z$ by $\R$ in the statement, thus reducing to prove the following. Let $\chi \colon N \to \mathbb{R}$ be a homogeneous quasimorphism. Suppose that $N$ can be expressed as a directed union of subgroup $\{N_i\}_{i \in I}$, and that for each $i \in I$ there is a homogeneous quasimorphism $X_i \colon G \to \mathbb{R}$ that extends $\chi|_{N_i}$, such that the defect of the $X_i$ is uniformly bounded. Then $\chi$ extends to a homogeneous quasimorphism $X \colon G \to \mathbb{R}$. Moreover, because $G$ is finitely generated, thus countable, we can assume that the directed set $I$ is just $\mathbb{N}$ with its well order.

    Let us make a further reduction: it suffices to show that we can modify each $X_i$ to obtain $X_i' \colon G \to \mathbb{R}$, which still extends $\chi|_{N_i}$ and has uniformly bounded defect, and moreover, for every $g\in G$, the sequence $\{|X_i'(g)|\}_{i\in \mathbb{N}}$ is bounded. Indeed, assuming this, fix a non-principal ultrafilter $\omega$ on $\mathbb{N}$. Then we can set $X(g) \coloneqq \lim_\omega X_i(g)$, which is a well-defined real number. Moreover, it is a homogeneous quasimorphism, since ultralimits are norm non-increasing and linear.

    Now consider the map $\alpha \colon G \to H_1(G; \mathbb{Q})$: this factors every homomorphism from $G$ to $\mathbb{Q}$, and its kernel coincides with the set of elements with a power that belongs to $[G, G]$. In particular, if $g \in \ker(\alpha)$, then its stable commutator length $\mathrm{scl}(g)$ is finite, and by Bavard duality \cite{bavard} we have a bound $|X_i(g)| \leq 2 \mathrm{scl}(g) D(X_i)$. This shows that for $g \in \ker(\alpha)$, we do not need to change $X_i$; note that this concludes the proof in the case in which $G$ has finite abelianisation (e.g. for mapping class groups).

    For the general case, let $x_1, \ldots, x_n$ be elements of $G$ that are mapped to a basis of $H_1(G; \mathbb{Q})$, chosen in such a way that $x_1, \ldots, x_m \notin N$ while $x_{m+1}, \ldots, x_n \in N$. Now, for each $j = 1, \ldots, m$, let $\lambda_j \colon G \to \mathbb{Q}$ be the functional dual to $x_j$; that is, if $g = \sum a_i x_i \in H_1(G; \mathbb{Q})$ then $\lambda_j(g) = a_j$. For each $i \in \mathbb{N}$, we set
    \[X_i' \coloneqq X_i - \sum\limits_{j = 1}^m X_i(x_j) \cdot \lambda_j.\]
    As a result, $X_i'$ is a homogeneous quasimorphism with the same defect as $X_i$, the same restriction on $N$ (in particular it is still an extension of $\chi|_{N_i}$) and moreover it vanishes on $x_j$ for all $j = 1, \ldots, m$. Now for an element $g \in G$, we can write $\alpha(g) = \sum a_j \alpha(x_j)$, for some $a_j \in \mathbb{Q}$. Thus, for a high enough power $p$, we can write
    \[\alpha(g^p) = \alpha\left( \prod\limits_{j = 1}^n x_j^{p_j} \right)\]
    for some integers $p_j \in \mathbb{Z}$. This way we have an expression $g^p = x y z$, where $x$ is a product of powers of $x_1, \ldots, x_m$; $y$ belongs to $N$, and $z$ belongs to $\ker(\alpha)$. By the first case treated before, $|X_i'(z)|$ is uniformly bounded. Moreover, $|X_i'(y)|$ is uniformly bounded because $y \in N_i$ for $i$ large enough, at which point $X_i(y) = \chi(y)$. Finally, $|X_i'(x)|$ is uniformly bounded, since $x$ is a product of $m$ terms, each of which vanish under $X_i'$. All in all, this shows that $|X_i'(g^p)|$ is uniformly bounded, and thus $|X_i(g)| = \frac{1}{p} |X_i(g^p)|$ is uniformly bounded, which concludes the proof.
\end{proof}

\begin{rem}
    As pointed out in Remark~\ref{rem:extendable_boundary}, in the setting of Proposition~\ref{bounded iff extendable} there is a single bounded cocycle that simultaneously extends all $\delta \chi_i|_{N_i}$. It could be possible to exploit this to find the extensions $X_i$ with uniformly bounded defect.
\end{rem}

\section{Hierarchical hyperbolicity of central extensions}\label{sec:HHG}
In this Section we explore the connections between boundedness and hierarchical hyperbolicity for central extensions. We first recall the notion of hierarchically hyperbolic spaces and groups.
\begin{defn}[Hierarchically hyperbolic space]\label{defn:HHS}
Let $\delta>0$ and $\cal{X}$ be a $(\delta,\delta)$--quasigeodesic space.  A \emph{hierarchically hyperbolic space (HHS) structure with constant $\delta$} for $\cal{X}$ is the data of an index set $\frakS$ and a set $\{ \C W \colon W\in\frakS\}$ of $\delta$-hyperbolic spaces $(\C W,\dist_W)$ such that the following axioms are satisfied.  \begin{enumerate}[label=(\arabic*{})]
    \item\textbf{(Projections.)}\label{axiom:projections} For each $W \in \frakS$, there exists a \emph{projection} $\pi_W \colon \cal{X} \rightarrow 2^{\C W}$  that is a $(\delta,\delta)$--coarsely Lipschitz, $\delta$--coarsely onto, $\delta$--coarse map.
		
    \item \textbf{(Nesting.)} \label{axiom:nesting} If $\frakS \neq \emptyset$, then $\frakS$ is equipped with a  partial order $\nest$ and contains a unique $\nest$--maximal element, denoted by $S$. When $V\nest W$, we say $V$ is \emph{nested} in $W$.  For each $W\in\frakS$, we denote by $\frakS_W$ the set of all $V\in\frakS$ with $V\nest W$.  Moreover, for all $V,W\in\frakS$ with $V\propnest W$ there is a specified non-empty subset $\rho^V_W\subseteq \C W$ with $\diam(\rho^V_W)\leq \delta$.
        
    \item \textbf{(Finite complexity.)} \label{axiom:finite_complexity} Any set of pairwise $\nest$--comparable elements has cardinality at most $\delta$.
            
    \item \textbf{(Orthogonality.)} \label{axiom:orthogonal} The set $\frakS$ has a symmetric relation called \emph{orthogonality}. If $V$ and $W$ are orthogonal, we write $V\perp W$ and require that $V$ and $W$ are not $\nest$--comparable. Further, whenever $V\nest W$ and $W\perp U$, we require that $V\perp U$. We denote by $\frakS_W^\perp$ the set of all $V\in \frakS$ with $V\perp W$.

    \item \textbf{(Containers.)} \label{axiom:containers}  For each $W \in \frakS$ and $U \in \frakS_W$ with $ \frakS_W\cap \frakS_U^\perp \neq \emptyset$, there exists $Q \in\frakS_W-\{W\}$ such that $V \nest Q$ whenever $V \in\frakS_W \cap \frakS_U^\perp$.  We call $Q$ the \emph{container of $U$ in $W$}.
		
    \item \textbf{(Transversality.)}\label{axiom:transversality} If $V,W\in\frakS$ are not orthogonal and neither is nested in the other, then we say $V$ and $W$ are \emph{transverse}, denoted $V\trans W$.  Moreover, for all $V,W \in \frakS$ with $V\trans W$, there are non-empty sets $\rho^V_W\subseteq \C W$ and $\rho^W_V\subseteq \C  V$, each of diameter at most $\delta$.

    \item \textbf{(Consistency.)} \label{axiom:consistency} For all $x \in\cal X$ and $U,V,W\in\frakS$:
		\begin{itemize}
			\item  if $V\trans W$, then $\min\left\{\dist_{W}(\pi_W(x),\rho^V_W),\dist_{V}(\pi_V(x),\rho^W_V)\right\}\leq \delta$,
			\item if $U\nest V$ and either $V\propnest W$, or $V\trans W$ and $W\not\perp U$, then $\dist_W(\rho^U_W,\rho^V_W)\leq \delta$.
		\end{itemize}

\item \textbf{(Bounded geodesic image (BGI).)} \label{axiom:bounded_geodesic_image} For all  $V,W\in\frakS$ and for all $x,y \in \cal{X}$, if  $V\propnest W$ and $\dist_V(\pi_V(x),\pi_V(y)) \geq \delta$, then every $\C W$--geodesic from $\pi_W(x)$ to $\pi_W(y)$ must intersect $\cal{N}_\delta(\rho_W^V)$.

    \item \textbf{(Large links.)} \label{axiom:large_link_lemma} For all $W\in\frakS$ and  $x,y\in\cal X$, there exists a collection $\{V_1,\dots,V_m\}\subseteq\frakS_W -\{W\}$ such that $m\le \delta \dist_{W}(\pi_W(x),\pi_W(y))+\delta$, and for all $U\in\frakS_W - \{W\}$, either $U\nest V_i$ for some $i$, or $\dist_{U}(\pi_U(x),\pi_U(y)) \leq \delta$.

    \item \textbf{(Partial realization.)} \label{axiom:partial_realisation}  If $\{V_i\}$ is a finite collection of pairwise orthogonal elements of $\frakS$ and $p_i\in  \C V_i$ for each $i$, then there exists $x\in \cal X$ \emph{realising} the tuple $(p_i)$, meaning that, for every $i$ and every $W\in \frakS$:
		\begin{itemize}
			\item $\dist_{V_i}(\pi_{V_{i}}(x),p_i)\leq \delta$;
			\item if $V_i\propnest W$ or $W\trans V_i$ then $\dist_{W}(\pi_W(x),\rho^{V_i}_W)\leq \delta$.
		\end{itemize}

    \item\textbf{(Uniqueness.)} There exists a function $\theta \colon [0,\infty) \to [0,\infty)$ so that for all $r \geq 0$, if $x,y\in\cal X$ and $\dist_\cal{X}(x,y)\geq\theta(r)$, then there exists $W\in\frakS$ such that $\dist_W(\pi_W(x),\pi_W(y))\geq r$. \label{axiom:uniqueness}
\end{enumerate}

    We use $\frakS$ to denote the HHS structure.  We call an element $U\in \frakS$ a \emph{domain}, the associated space $\cal CU$ its \emph{coordinate space}, and call the maps $\rho_W^V$ the \emph{relative projections} from $V$ to $W$. The quantity $\delta$ is called a \emph{hierarchy constant} for $\frakS$. We often suppress reference to the projection maps, so for every $x,y\in\cal X$ and $U\in \frakS$ we write $\dist_U(x,y)$ to mean $\dist_U(\pi_U(x), \pi_U(y))$.
\end{defn}

\begin{defn}[Hierarchically hyperbolic group]\label{defn:HHG}
A finitely generated group $G$ is a \emph{hierarchically hyperbolic group} (HHG) if the following hold.
    \begin{enumerate}[label=(\roman*{})]
        \item\label{HHG_structure} $G$ acts metrically properly and coboundedly on a space $\cal X$ admitting a HHS structure $\frakS$.
        \item\label{HHG_action} There is a $\nest$--, $\perp$--, and $\trans$--preserving action of $G$ on $\frakS$ by bijections such that $\frakS$ contains finitely many $G$--orbits.
        \item\label{HHG_isometries} For each $W\in \frakS$ and $g\in G$, there exists an isometry $g_W\colon \cal CW \to \cal C(gW)$ satisfying the following for all $V,W\in\frakS$ and $g,h\in G$.
            \begin{itemize}
                \item The maps $(gh)_W\colon \cal CW \to \cal C(ghW)$ and $g_{hW}\circ h_W\colon \cal CW\to \cal C(hW)$ coincide.
                \item For each $x\in\cal  X$, $g_W(\pi_W(x))=\pi_{gW}(g\cdot x)$ in $\cal C(gW)$.
                \item If $V\trans W$ or $V\propnest W$, then $g_W(\rho^V_W)=\rho^{gV}_{gW}$ in $\cal C(gW)$.
            \end{itemize}
    \end{enumerate}
     We often drop the indices and denote each $g_W$ simply by $g$. When the underlying HHS is not relevant, we denote a HHG by $(G,\frakS)$.
\end{defn}

\begin{rem}[On finite generation]
    Our results in Section \ref{sec:background} only required finite generation of the kernel, not of the quotient or of the extension. However, note that HHGs are finitely generated by definition.
\end{rem}

\begin{rem}[Moral compass]
    When reading the definitions above, the uninitiated reader should keep in mind the motivating example of a HHG, which is the mapping class group of an orientable, finite-type surface. In this context, $\cal X$ is the marking complex from \cite[Section 2.5]{masurminsky2}; the elements of $\frakS$ are isotopy classes of subsurfaces, with nesting given by inclusion and orthogonality corresponding to disjointness (both up to isotopy); finally, the coordinate space associated to a subsurface is the corresponding curve graph, onto which $\cal X$ maps via the subsurface projection. The various axioms from Definition~\ref{defn:HHS} are abstractions of the properties of curve graphs and clean markings from \cite{masurminsky1,masurminsky2}.
\end{rem}

\subsection{Variations on the axioms}
\begin{rem}[Normalisation] 
    For experts, Definition~\ref{defn:HHS} is that of a \emph{normalised} HHS, since the original definition \cite{HHS_II} only required the coordinate projections to have uniformly quasiconvex images. However it is always possible to normalise a HHS structure, by restricting each coordinate space $\C U$ to a neighbourhood of $\pi_U(\cal X)$ (see e.g. \cite[Remark 1.3]{HHS_II}).
\end{rem}

The following is a strengthening of Axiom~\ref{axiom:containers}:

\begin{defn}\label{def:CC}
    A hierarchically hyperbolic space $(\cal X, \frakS)$ has \emph{clean containers} if for each $W \in \frakS$ and $U \in \frakS_W$ with $ \frakS_W\cap \frakS_U^\perp \neq \emptyset$, there exists $Q \in\frakS_W-\{W\}$ such that $V \nest Q$ whenever $V \in\frakS_W \cap \frakS_U^\perp$, \emph{and $Q\orth U$}.
\end{defn}

\begin{rem}[Bounded domain dichotomy]\label{rem:BDD}
A HHG acts cofinitely on the set of domains, and two domains in the same orbit have isometric coordinate spaces. Therefore, up to enlarging $\delta$, we can and will assume that every $U\in \frakS$ is either \emph{unbounded} (meaning that $\diam\C U=\infty$) or $\diam\C U\le \delta$. 
\end{rem}

\begin{rem}[Passing up]\label{rem:passingup}
    The large links Axiom~\ref{axiom:large_link_lemma} can be replaced by the following weaker requirement:
    \begin{enumerate}[start=9, label=(\arabic*{}')]
        \item \textbf{(Passing up.)}\label{axiom:passingup} For every $t>0$, there exists an integer $P=P(t)>0$ such that if $V\in \frakS$ and $x,y\in \cal X$ satisfy $\dist_{U_i}(x,y)>\delta$ for a collection of domains $\{U_i\}_{i=1}^P$ with $U_i\in \frakS_V$, then there exists $W\in \frakS_V$ containing some $U_i$, and such that $\dist_W(x,y)>t$. We call $P\colon(0,\infty)\to (0,\infty)$ the \emph{passing up function}.
    \end{enumerate}
It was shown in \cite[Lemma~2.5]{HHS_II} that every HHS satisfies the Passing up axiom; conversely, in \cite[Section~4.8]{Durham_cubinf} it is argued that the passing up axiom, together with BGI and normalisation, imply the large links axiom. \end{rem}

\begin{rem}\label{rem:checking_passingup}
    Notice that, when one has to verify the Passing up axiom, it is enough to define $P(t)$ only for $t\ge \delta$, because then one can set
$$P'(t)=\begin{cases}
    P(\delta)\mbox{ if }t\le\delta;\\
    P(t)\mbox{ if }t\ge \delta.
\end{cases}$$
In particular, if the candidate collection of domains has the bounded domain dichotomy, one can further require the domains $\{U_i\}_{i=1}^P$ and $W$ to be unbounded.
\end{rem}

\subsection{Action of the centre}
Here we prove several results about the action of the centre of a HHG on the structure. Firstly, we recall a definition:

\begin{defn}[{Big domains \cite{DHS}}]
Let $(\cal X,\frakS)$ be a HHG structure for $G$, fix a basepoint $x\in \cal X$, and let $z\in G$. The \emph{bigset} of $z$ is 
    $$\Bigdom{z}=\{U\in \frakS\,|\,\diam\pi_U(\langle z\rangle\cdot  x)=\infty\}.$$
 It is clear from the definition that $\Bigdom{z}$ does not depend on the choice of the basepoint; furthermore, $\Bigdom{z}\neq \emptyset$ if and only if  $z$ has infinite order \cite[Proposition 6.4]{DHS}, and in this case $\Bigdom{z}$ is a finite collection of pairwise orthogonal domains (this follows from combining \cite[Lemma 6.7]{DHS} with \cite[Lemma 2.1]{HHS_II}). 
\end{defn}

\begin{lemma}\label{lem:central_action}
    Let $(\cal X, \frakS)$ be a HHG structure for a group $E$, and let $K$ be finitely generated and contained in the centre of $E$. Then:
    \begin{enumerate}
        \item\label{item:K_fixes_all} $K$ fixes every unbounded domain.
        \item\label{item:K_acts_unif_bdd} Let $\{z_1,\ldots, z_n\}$ be a basis for $K$. Then there exists $C>0$ such that, for every $x\in \cal X $ and every $W\in \frakS-\bigcup_{i=1}^n \Bigdom{z_i}$, $\diam \pi_W(K\cdot x)<C$.
    \end{enumerate}
\end{lemma}

\begin{proof}
    \eqref{item:K_fixes_all} Let $\delta$ be a HHG constant for the structure. We first argue that the $K$-action on $\frakS$ fixes every unbounded domain. Fix $x_0\in \cal X$, $k\in K-\{1\}$, and an unbounded domain $W\in\frakS$. For every $g,h\in E$,
    $$\dist_{kW}(hx_0, ghx_0)=\dist_{W}(k^{-1}hx_0, k^{-1}ghx_0);$$
    then the triangle inequality, plus the fact that $k$ is central in $E$, yield that
    \begin{align*}
        \left|\dist_{W}(k^{-1}hx_0, k^{-1}ghx_0)-\dist_{W}(hx_0, ghx_0)\right|&\le \dist_{W}(hx_0, k^{-1}hx_0)+\dist_{W}(ghx_0, k^{-1}ghx_0)\\
        &=\dist_{h^{-1}W}(x_0, k^{-1}x_0)+\dist_{(gh)^{-1}W}(x_0, k^{-1}x_0)\\
        &\le 2\sup_{V\in \frakS}\dist_{V}(x_0, k^{-1}x_0).
    \end{align*}
    The quantity $M\coloneq 2\sup_{V\in \frakS}\dist_{V}(x_0, k^{-1}x_0)$ is bounded in terms of $\dist_{\cal X}(x_0, k^{-1} x_0)$, since coordinate projections are uniformly coarsely Lipschitz. We just proved that projection distances in $W$ and in $kW$ must agree up to an error of $M$. \\
    Now we show that $kW=W$, by excluding every other possibility.
    \begin{itemize}
        \item If, say, $W\propnest kW$ and $k$ has infinite order, then we would get an infinite chain of properly nested domains $W\propnest kW\propnest k^2W\ldots$, contradicting finite complexity; if instead $k$ has order $n<\infty$, then the above chain would eventually yield that $W\propnest W$, a contradiction. The same argument excludes that $kW\propnest W$.
        \item If $kW\orth W$, since $\C W$ is unbounded, the partial realisation axiom~\ref{axiom:partial_realisation} can be used to find an element $g\in E$ such that $\dist_W(x_0,gx_0)\le \delta$, while $\dist_{kW}(x_0,gx_0)\ge M+\delta+1$. This would contradict the fact that $|\dist_W(x_0,gx_0)-\dist_{kW}(x_0,gx_0)|\le M$.
        \item Finally, suppose $W\trans kW$. Since $\C W$ is unbounded, we can find elements $g,h\in E$ such that $$\min\{\dist_W(\rho^{kW}_W, hx_0), \dist_W(\rho^{kW}_W, ghx_0), \dist_W(hx_0, ghx_0)\}\ge M+3\delta+1.$$ 
        However the consistency axiom~\ref{axiom:consistency} would give that 
        $$\dist_{kW}(hx_0, ghx_0)\le \dist_{kW}(hx_0, \rho^{W}_{kW})+\dist_{kW}(\rho^{W}_{kW}, ghx_0)+\diam\rho^{W}_{kW}\le 3\delta, $$
        again contradicting that $|\dist_W(hx_0,ghx_0)-\dist_{kW}(hx_0,ghx_0)|\le M$.
    \end{itemize}

    \eqref{item:K_acts_unif_bdd} Let $\cal U=\bigcup_{i-1}^n \Bigdom{z_i}$ and $\wfS=\frakS-\cal U$. Notice first that $E$ permutes each $\Bigdom{z_i}$, since for every $h\in E$ and $x\in \cal X$ we have
    $$\diam\pi_{hU}(\langle z_i\rangle\cdot x)=\diam\pi_{U}(h^{-1}\langle z_i\rangle\cdot x)=\diam\pi_{U}(\langle z_i\rangle h^{-1}\cdot x)=\infty.$$
    Thus there exists a finite-index, normal subgroup $E'\le E$ fixing every $U\in \cal U$. Furthermore, again because $K$ is central, every element in $E'$ pointwise fixes the ideal endpoints of an axis for $z_i$ in $\C U$ whenever $U\in \Bigdom{z_i}$, so $\C U$ is a quasiline by \cite[Proposition 3.3]{ANS:UEG}. Now equip $E'$ with the same HHG structure as $E$, which we can do as the $E'$-action on $\frakS$ is still cofinite. If we apply \cite[Lemma 3.1, Proposition 3.2, and Proposition 3.4]{ANS:UEG} to $(E',\frakS)$, we get that that each $U\in \cal U$ is orthogonal to any other unbounded domain in $\frakS$; in particular, any two $U,U'\in \cal U$ are orthogonal. 
    \\
     We now move to the proof of \eqref{item:K_acts_unif_bdd}. We first prove that a specific $K$-orbit in $\cal X$ has uniformly bounded projections to all $W\in \wfS$. Complete $\cal U$ to a maximal collection $\cal V$ of pairwise orthogonal, unbounded domains, each of which does not contain other unbounded domains. Moreover, let $x_0\in \cal X$ be any point given by the partial realisation axiom~\ref{axiom:partial_realisation} applied to the above collection, with any choice of coordinates $p_i\in\C V_i$. Every other unbounded domain $W\in \wfS$ must contain, or be transverse to, some $V\in \cal V$, so $\pi_W(x_0)$ is uniformly close to $\rho^{V}_W$. But $K$ fixes both $V$ and $W$ by (1), so every $K$-translate of $x_0$ projects uniformly close to $\rho^{V}_W$. This shows that $\diam\pi_W(K\cdot x_0)$ is uniformly bounded whenever $W\in \wfS-\cal V$. Furthermore, for every $V\in \cal V-\cal U$, $\langle z_1\rangle\cdot x_0$ has bounded projection to $\C V$; then $\langle z_1, z_2\rangle\cdot x_0=\langle z_2\rangle\cdot (\langle z_1\rangle\cdot x_0)$ has bounded projection to $\C V$ as well, and analogously we get that $\diam_V K\cdot x_0<\infty$. Since there are finitely many $V_i$, we just proved that $\sup_{U\in \wfS}\diam\pi_U(K\cdot x_0)$ is bounded by some constant $C'$.
     \\
    Now notice that
    \begin{align*}
        \sup_{g\in E',\, U\in \wfS}\diam\pi_U(K\cdot gx_0)&=\sup_{g\in E',\, U\in \wfS}\diam\pi_{g^{-1}U}(K\cdot x_0)\\
        &=\sup_{U\in \wfS}\diam\pi_U(K\cdot x_0)\le C',
    \end{align*}
    where we used that every $g\in E'$ commutes with $K$ and preserves $\wfS$. Since $E'$ acts coboundedly on $\cal X$, and since projections are uniformly coarsely Lipschitz, the above bound implies Item (2).
\end{proof}

\subsection{Bounded central extensions and hierarchical hyperbolicity}
We now move to the cohomological characterisation of when a central extension of a HHG, with finitely generated kernel, is a HHG.

\begin{thm}\label{thm:HHG_iff_bounded_abelian_ker}
    Let $G$ be a HHG, and let  $1 \to K \to E \to G \to 1$ be a central extension with finitely generated kernel. Then $E$ is a HHG if and only if the extension is bounded.
\end{thm}

\begin{rem} As we shall expand upon in the proof, the fact that a bounded $\Z$-central extension of a HHG is a HHG was essentially proven in \cite[Proposition 5.14]{uniform_undistortion}.  In turn, the argument there refined \cite[Corollary 4.3]{HRSS_3manifold}, which dealt with the case with hyperbolic base and only relied on the fact that every $\Z$-central extension of a hyperbolic group is bounded \cite{neumannreeves}. As we have seen in Example~\ref{Heisenberg}, this is no longer automatically true for general HHG. Even assuming acylindrical hyperbolicity of the base is not enough: indeed $H^2(\mathbb{Z}^2 \ast \mathbb{Z}; \mathbb{Z}) \cong H^2(\mathbb{Z}^2; \mathbb{Z})$, which produces an unbounded class. The argument is the same as Example \ref{Heisenberg}, using for instance that the comparison map $H^2_b(\mathbb{Z}^2 \ast \mathbb{Z}; \mathbb{R}) \to H^2(\mathbb{Z}^2 \ast \mathbb{Z}; \mathbb{R})$ is trivial \cite[Example 4.7]{kevin}.
\end{rem}

\begin{rem}
    \label{rem:QITB}
    It is natural to ask whether one can also characterise the extension being just HHS rather than HHG. A sufficient, and potentially necessary condition for this is that the extension is quasi-isometrically trivial, which turns out to be equivalent to the Euler class of the extension being \emph{weakly bounded}, in the sense of \cite{neumannreeves}. For many groups (for instance right-angled Artin groups), weakly bounded classes are bounded \cite{frigsisto}, and this is called property QITB (that is, quasi-isometrically trivial implies bounded). However, there are examples of quasi-isometrically trivial extensions which are not bounded \cite{frigsisto,ascarimilizia}. To summarise, two intriguing questions arise from this:
    \begin{enumerate}
        \item If a central extension of an HHG is an HHS, is the extension quasi-isometrically trivial?
        \item\label{do_hhg_have_QITB} Do HHGs satisfy QITB? Equivalently, in view of Theorem~\ref{thm:HHG_iff_bounded_abelian_ker}, is a quasi-isometrically trivial extension of a HHG also a HHG?
    \end{enumerate}
\end{rem}

\begin{proof}[Proof of Theorem~\ref{thm:HHG_iff_bounded_abelian_ker}]
Let $K \cong\Z^n\times T$ for some $n\ge 0$ and $T$ finite.

$(\Leftarrow)$ Assume that the extension $E \to G$ is bounded. A finite extension of a HHG is a HHG via the geometric action on the same space (this is true even when the extension is not central); thus we can reduce to the case where $T=\{0\}$. Moreover, by Corollary \ref{cor:decomposing:bounded:extensions} and induction on $n$, we can also assume that $K\cong \Z$. 

Now, since $E\to G$ is bounded, Proposition~\ref{bounded via qm:abelian_ker} produces a quasimorphism $\phi\colon E\to K$ which is the identity on $K$. In turn, by \cite[Lemma 4.15]{ABO} there exists a Cayley graph $L$ for $E$ such that $\phi\colon L\to \R$ is a quasi-isometry. In particular $L$ is a quasiline on which $K$ acts loxodromically. Then the collection $\{E,G,L\}$ satisfies the assumptions of \cite[Proposition 5.14]{uniform_undistortion}, which equips $E$ with a HHG structure.

\medskip
$(\Rightarrow)$ Now assume that $E$ admits a HHG structure, coming from the action on some HHS $(\cal X,\frakS)$. Our goal is to produce a quasihomomorphism $\Psi\colon E\to K$ which is the identity on $K$, and then we will conclude by Proposition~\ref{bounded via qm:abelian_ker}. Since this is obvious if $K$ is finite, let us assume for the rest of the proof that $n \geq 1$.

Fix a basepoint $x\in \cal X$ and a basis $z_1,\ldots,z_n$ for $K$, which we order in such a way that the first $m$ elements generate a subgroup isomorphic to $\Z^m$ and the other elements have finite order. Let $\cal U=\bigcup_{i=1}^m\Bigdom{z_i}$, and let $E'$ be the finite-index, normal subgroup of $E$ fixing each $U\in \cal U$, as in the proof of Lemma~\ref{lem:central_action}. Notice that, by the same Lemma, $K\le E'$; furthermore, $(\cal X, \frakS)$ is also a HHG structure for $E'$.

By arguing as in \cite[Lemma 4.20]{ABO}, one can show that $\C U$ is a quasiline for every $U\in \cal U$. Indeed, $K\le E'$ has unbounded orbits on $\C U$, so $E'$ does not act elliptically; moreover, the action is not parabolic as it is cobounded (this is because $\pi_U$ is coarsely surjective). By the classification of group actions on hyperbolic spaces (which dates back to Gromov \cite{gromov}, see e.g. \cite[Theorem 4.2]{ABO} for a more modern statement), there must be a loxodromic element $h\in E'$. As $K$ commutes with $h$, it must fix the endpoints $h^+,h^-\in\partial \C U$. In particular, the action of $K$ is not parabolic, so some $z\in K$ is a loxodromic isometry of $\C U$ with endpoints $h^\pm$. Again, this means that every element in $E'$, which commutes with $z$, must fix both $h^+$ and $h^-$; hence, as the action is cobounded, $\C U$ must be a quasiline. 

The above argument also shows that $E'$ acts on $\C U$ coboundedly and without inversions. In particular, we can $E'$-equivariantly replace each $\C U$ with a Cayley graph for $E'$, so that the projection $E'\to \C U$ is the identity at the level of vertices.

Now let $x_0\in \C U$ correspond to the identity element of $E$. Let $\phi_U\colon E'\to \R$ be the Busemann quasimorphism associated to the $E'$-action on $\C U$, as in Example~\ref{ex:busemann}. Since the absolute value of $\phi_U(g)$ coarsely coincides with the distance in $\C U$ from the basepoint $x_0$, we see that $\phi_U$ is a quasi-isometry when seen as a map from $\C U$ to $\R$.

Going back to $K$, Lemma~\ref{lem:central_action} tells us $K$-orbits are uniformly bounded in every domain outside $\cal U$. By the Distance Formula \cite[Theorem 4.5]{HHS_II}, this means that orbit maps quasi-isometrically embed $K$ inside the product $\prod_{U\in \cal U} \C U$; in turn, since every $\phi_U\colon \C U\to \R$ is a quasi-isometry, we get that the map $\Phi\colon K\to \R^{|\cal U|}$ mapping $k\in K$ to $(\phi_U(k))_{U\in \cal U}$ is a quasi-isometry. Since every homogeneous quasimorphism on the Abelian group $K$ is a homomorphism \cite[Corollary 2.12]{frigerio}, the image of $\Phi$ must span an $m$-dimensional subspace. Hence, if we identify $\langle z_1,\ldots, z_m\rangle$ with $\Z^m\le \R^m$, there is a linear map $M\colon \R^{|\cal U|}\to \R^m$ such that $\Psi\coloneq M\circ \Phi\colon E'\to \R^m$ is the identity on $\Z^m$. Notice that $\Psi$ is a quasihomomorphism as it is a composition of the quasihomomorphism $\Phi$ and a linear map; however, it is not yet the quasihomomorphism we are looking for, since it is only defined on $E'$.

To solve the above problem, let $g_1,\ldots,g_r\in E$ be representatives for the cosets of $E'$, and define a quasihomomorphism $\Theta\colon E'\to \R^m$ by mapping $h\in E'$ to
$$\Theta(h)=\frac{1}{r}\sum_{i=1}^r \Psi(g_ihg_i^{-1}).$$
Notice that $\Theta$ is still the identity on $\Z^m$, since the latter is central in $E$; furthermore, by arguing as in e.g. \cite[Lemma 3.4]{short_HHG:I}, one can see that $\Theta$ is now almost invariant under conjugation by $E$, so let $D>0$ be such that, for every $e\in E$ and $h\in E'$, $|\Theta(h)-\Theta(ehe^{-1})|\le D$. 

Finally, let define $\Xi\colon E\to \R^m$ by setting $\Xi(g_ih)=\Theta(h)$ for every $h\in E'$ and $i=1,\ldots,r$. We now prove that $\Xi$ is a quasihomomorphism. Let
$$\cal E=\{e\in E'\mid \exists 1\le i,j,k\le r\mbox{ s.t. }g_ig_j=g_k e\}.$$
Then, for every $g,g'\in \{G_1,\ldots, g_r\}$ and every $h,h'\in E'$, we have
$$ghg'h'=g''e((g')^{-1}hg')h',$$
where $g''\in \{g_1,\ldots,g_r\}$ and $e\in \cal E$ are such that $gg'=g''e$. Hence, if the defect of $\Theta$ is $C>0$, we have that
\begin{align*}
    &\left |\Xi(gh)+\Xi(g'h')-\Xi(ghg'h')\right|\\
    =&\left |\Theta(h)+\Theta(h')-\Theta(e((g')^{-1}hg')h')\right|\\
    \le&\left |\Theta(h)+\Theta(h')-\Theta(e)-\Theta((g')^{-1}hg')-\Theta(h')\right|+2C\\
    \le&\left |\Theta(h)+\Theta(h')-\Theta(e)-\Theta(h)-\Theta(h')\right|+2C+D\\
    \le&|\Theta(e)|+4C+D\\
    \le&\max_{e\in \cal E}|\Theta(e)|+4C+D.
\end{align*}
Up to taking the integer part, we can assume that $\Xi$ takes values in $\Z^m$, and up to a further bounded modification we can assume that $\Xi|_K$ is the identity. As a bounded modification of a quasihomomorphism with abelian target is still a quasihomomorphism, $\Xi$ satisfies the requirements of Proposition~\ref{bounded via qm:abelian_ker}, and the proof is complete.
\end{proof}

Next, we prove a statement in the opposite direction:

\begin{thm}\label{thm:centralquot_is_HHS}
    Let $1 \to K \to E \to G \to 1$ be a central extension with finitely generated kernel. If $E$ is a HHG, then the extension is bounded. Moreover, if $E$ has a HHG structure with clean containers, there exists a finite-index subgroup $E'\le E$ containing $K$ such that $E'/K$ is a HHG.
\end{thm}

\begin{proof}
    The proof of $(\Rightarrow)$ of Theorem \ref{thm:HHG_iff_bounded_abelian_ker} did not use that $G$ is a HHG, so it shows that if $E$ is HHG, then the extension is bounded. We now focus on proving the ``moreover'' part, whose proof we split into several steps.

\par\medskip\textbf{Step 1: choice of the finite-index subgroup.} Consider a HHG structure $(\mathcal X, \frakS)$ for $E$ with clean containers. Fix a basis $\{z_1,\ldots, z_n\}$ for $K$, and let $$\cal U=\bigcup_{i=1}^n \Bigdom{z_i}.$$ As in the proof of Theorem~\ref{thm:HHG_iff_bounded_abelian_ker}, $E$ acts on each $\Bigdom{z_i}$, so let $E'$ be the finite-index subgroup of $E$ that fixes each $U\in \cal U$; notice that $K\le E'$ by Lemma~\ref{lem:central_action}. Since $E'$ has finite index, $(\cal X,\frakS)$ is a HHG structure for $E'$, so we can identify $\cal X$ with a Cayley graph for $E'$ with respect to a finite generating set. For later purposes, let $\delta$ be a HHG constant for the structure.

 \par\medskip
\textbf{Step 2: removing unnecessary bounded domains.} 
In this step we restrict the domain set of the HHG structure. Firstly, let 
$$\frakS'=\{W\in \frakS\mid \exists V\nest W\mbox{ s.t. } \mathcal C V \mbox{ unbounded}\}.$$
We claim that $(E',\frakS')$ is again a HHG structure with clean containers, with the same domains and coordinate spaces. We point out that $\frakS'$ is $E'$-invariant and contains all unbounded domains; therefore, by inspection of Definition~\ref{defn:HHS}, removing some bounded domains can only affect the existence of containers and the validity of the large link axiom, so we must check that both still hold in $\frakS'$.

    \begin{itemize}
\item \textbf{Clean containers:} Let $V,W\in \frakS'$ be such that $V\propnest W$, and there exists $U\in \frakS'$ which is nested in $W$ and orthogonal to $V$. By the containers axiom in $\frakS$, there exists $T\in \frakS$ such that $T\propnest W$, $T\orth V$, and $T$ contains all $U$ as above. Since $U\in \frakS'$, it contains an unbounded domain, and in turn this means that $T\in \frakS'$. Hence $T$ is the container for $V$ inside $W$ in $\frakS'$.

\item \textbf{Large links:} By Remark~\ref{rem:passingup}, it is enough to check the passing up axiom~\ref{axiom:passingup}. In turn, since $\frakS'\subseteq \frakS$ has the bounded domain dichotomy, by Remark~\ref{rem:checking_passingup} it is enough to check the passing up axiom for unbounded domains, which already held in  $\frakS$.
\end{itemize}

Now, by \cite[Lemma 3.1 and Proposition 3.4]{ANS:UEG}, each $U\in \cal U$ is orthogonal to any other unbounded domain, so in $\frakS'$ there are no domains which are properly nested in any $U\in \cal U$. Let $\frakS''$ be obtained from $\frakS'$ by removing all domains which are transverse to some $U\in \cal U$, so that $\frakS''$ is an $E'$-invariant collection which still contains all unbounded domains. As above, $(E',\frakS'')$ is a HHG structure if we check that the clean containers axiom still holds. To this extent, let $V\propnest W$ be domains in $\frakS''$ for which a container is needed; in particular $W\not\in \cal U$ as the domains in $\cal U$ are all $\nest$-minimal. Let $T\in \frakS'$ be the clean container for $V$ in $W$. Suppose by contradiction that $T\trans U$ for some $U\in \cal U$. Then $U\propnest W$, because if $U\orth W$ we would have that $T\orth U$ as well. Furthermore, $U\orth V$, because if $U\nest V$ then $T$ would be orthogonal to $U$ (here we are crucially using that $T\orth V$ because $\frakS'$ has clean containers). However this is a contradiction, as then $T$ should contain $U$.

\par\medskip \textbf{Interlude: Reduction to cyclic kernel.}
In view of the above discussion, we shall say that a collection $\{h_1,\ldots, h_n\}$ of elements of a HHG $(H,\frakS)$ is \emph{clean} if for every $W\in \frakS$, $i=1,\ldots,n$, and $U\in \Bigdom{h_1}$, either $U\nest W$ or $U\orth W$.
\\
Now let $z=z_1$ and $Z=\langle z\rangle$. Our goal is to prove that $E'/Z$ admits a HHG structure $\ofS$ such that the image of $\{z_2,\ldots, z_n\}$ is clean in $(E'/Z,\ofS)$. Therefore, if in the following arguments we are careful to never use that $\frakS''$ has clean containers, but only that $\{z_1,\ldots,z_n\}$ is clean in $(E',\frakS'')$, then we will be able to conclude by induction that $E'/K$ is a HHG, thus proving Theorem~\ref{thm:centralquot_is_HHS}

    \par\medskip
    \textbf{Step 3: reduction to a single big domain.}
    In this step we replace some coordinate spaces, while keeping the domain set $\frakS''$, to get a new HHG structure where $z$ has at most one big domain (this step is unnecessary if $z$ has finite order). Let $\Bigdom{z}=\{U_1,\ldots,U_j\}$, and for every $i=1,\ldots, j$ let $\phi_i\coloneq E'\to \R$ be the Busemann quasimorphism for the action on $\C U_i$. We first claim that these quasimorphisms are linearly independent. Indeed, if, say, $\phi_1$ was a linear combination of the others, then there would be some function $f\colon \R_{\ge0}\to \R_{\ge0}$ such that $\dist_{U_1}(g,1)\le f(\dist_{U_2}(g,1), \ldots, \dist_{U_j}(g,1))$ for every $g\in E'$ (here we used that, as argued in Example~\ref{ex:busemann}, the absolute value of each quasimorphism coarsely coincides with the distance in the corresponding quasiline). However the elements of $\Bigdom{z}$ are pairwise orthogonal, so the partial realisation axiom~\ref{axiom:partial_realisation} implies that $E'$ acts coboundedly on the product $\prod_{i=1}^j\C U_i$. In particular, one can find a sequence of $g_n\in E'$ for which $\dist_{U_1}(g_n,1)\to\infty $ while the other projections are uniformly bounded, contradicting the existence of $f$. 

    Next, set $\psi_1=\phi_1$, and for every $ 2\le i\le j$ let $\psi_i=\phi_i-\frac{\phi_i(z)}{\phi_1(z)}\phi_1$, which is non-trivial by linear independence. For every $1\le i\le j$ let $L_i$ be a Cayley graph for $E'$ such that $\psi_i\colon L_i\to \R$ is a quasi-isometry, as constructed in \cite[Lemma 4.15]{ABO}. Notice that $\psi_i(z)=0$ whenever $i\ge 2$, so $Z$ now acts elliptically on all quasilines excluding $L_1$. 

    Now, replacing $\C U_1$ by $L_1$ gives a new HHG structure for $E'$, since the two spaces are $E'$-equivariantly quasi-isometric. We then replace each $\C U_i$ by the corresponding $L_i$, with projection given by the identity map $\cal X\to L_i$ as both are Cayley graphs for $E'$. Notice that it is not necessary to define new relative projections, as no domain is transverse to, or nested in, any $U_i$.
    
    We claim that the above replacement procedure gives a HHG structure for $E'$, where $z$ will then have a single big domain. All requirements only involving the relations on $\frakS''$, or the projections between domains in $\frakS''$, still hold by construction, so we focus on the axioms which impose restrictions on the $L_i$.
    
    \begin{itemize}
        \item \textbf{Projections~\ref{axiom:projections}.} The projection $\cal X\to L_i$ is surjective and Lipschitz, as $\cal X$ is a Cayley graph for $E'$ with respect to a finite generating set.
        \item \textbf{Consistency~\ref{axiom:consistency}.} Since $\{z\}$ is clean in $\frakS''$, no element of $\frakS''$ is either transverse to, or nested in, any $U\in \Bigdom{z}$. Hence, both consistency inequalities only involve domains from $\frakS''-\Bigdom{z}$, and therefore already held in the original structure.
        \item \textbf{Bounded geodesic image~\ref{axiom:bounded_geodesic_image}.} Since every $U_i$ is orthogonal to every other unbounded domain, there are no pairs of nested unbounded domains involving $U_i$; all other instances of the axiom already appeared in the original structure.
        \item \textbf{Large links~\ref{axiom:large_link_lemma}.} As we only replaced finitely many coordinate spaces, the axiom still holds (possibly after replacing $\delta$ by a bigger constant).
        \item \textbf{Partial realisation~\ref{axiom:partial_realisation}.} Let $\cal V= \{(V_i,q_i\in \C V_i)\}_{i=1,\ldots,m}$ be a collection of pairwise orthogonal domains, each with a point in the associated coordinate space. Suppose first that no element in $\cal V$ belongs to $\Bigdom{z}$, and let $x\in \cal X$ be a realisation point for $\cal V$ in the original HHG structure. Since no element of $\cal V$ is transverse to, or nested in, any element of $\Bigdom{z}$, $x$ is also a realisation point in the new structure, and we are done.

        Thus, suppose $\Bigdom{z}$ intersects the support of $\cal V$ nontrivially. If $U_1$ appears in $\cal V$ then let $g_1$ be the coordinate $q_1\in L_1$ that we have to realise; otherwise let $g_1=1\in E'$. Whenever $U_i\in \cal V$ for some $i\ge 2$ let $q_i\in L_i$ the coordinate to realise, choose $g_i\in E'$ such that $\phi_i(g_i)=\psi_i(q_i)+\psi_1(g_1)$, and let $p_i=\pi_{\C U_i}(g_i)$.
        By the partial realisation axiom for the original structure, there exists some $x\in \cal X$ realising the collection $$\{(U_i, p_i\in \cal C U_i)\}_{U_i\in (\Bigdom{z}-\{U_1\})\cap \cal V}\cup \{(V_i,q_i\in \C V_i)\}_{V_i\in \cal V-(\Bigdom{z}-\{U_1\})},$$
        and we claim that $x$ is uniformly close to $p_i$ in each $L_i$. Since $x$ is a realisation point, for every $i=1,\ldots,j$ we have that $\dist_{\C U_i}(g_i,x)\le \delta$, so that $\phi_i(x)$ and $\phi_i(g_i)$ are uniformly close. In turn, this means that $\psi_i(x)$ is uniformly close to $\psi_i(p_i)$ for all $i$, and therefore $\dist_{L_i}(p_i,x)$ is uniformly bounded, as required.
        \item \textbf{Uniqueness~\ref{axiom:uniqueness}} Suppose $x,y\in \cal X$ have uniformly close projections to all domains in the new structure. If we show that their projections to each $\C U_i$ are also uniformly bounded then $x$ and $y$ are uniformly close, by the uniqueness axiom in the original structure. Up to the group action we can assume that $y=1$, so that $\dist_{L_i}(1,x)$ coarsely coincides with the absolute value of $\psi_i(x)$. Since all distances in the $L_i$ are uniformly bounded, so are the absolute values of $\psi_i$, hence of $\phi_i$, and again this means that $\dist_{\C U_i}(1,x)$ is uniformly bounded for all $i$. 
        \item \textbf{HHG structure.} All projections (old and new) are $E'$-equivariant, so we built a hierarchically hyperbolic \emph{group} structure.
    \end{itemize}
    Before proceeding, we point out that the collection $\{z_2,\ldots, z_n\}$ is still clean in the new structure. This is because we did not change the projections to any domain outside $\cal U$; hence, for every $i=2,\ldots,n$, $\Bigdom{z_i}$ must still be contained inside $\cal U$ (though it might differ from what it was before Step 3).

    \par\medskip\textbf{Step 4: HHG structure of the quotient.} 
    We are finally ready to prove that $E'/Z$ is a HHG, with the structure that we now describe. If $z$ has infinite order let $U_z$ be the unique big domain for $z$ in the new structure, and let $\wfS=\frakS''-\{U_z\}$; otherwise let $\wfS=\frakS''$. Let $\cal Y=\cal X/Z$ and $\ofS=\wfS/Z$. Given a domain $W\in \wfS$, let $[W]$ be its image in $\ofS$. Two domains $[U], [V]\in \ofS$ are orthogonal (resp. nested) if they admit orthogonal (resp. nested) representatives $U\in [U]$ and $V\in [V]$. Set 
    $$\cal C [W]=\left(\bigcup_{W\in [W]}\cal C W\right)/K,$$ and for any representative let $q_W\colon \cal C W\to \cal C [W]$ be the quotient projection. This map is $1$-Lipschitz; moreover, it is a uniform quality quasi-isometry, because $K$-orbits are uniformly bounded in $\cal C W$ by Lemma~~\ref{lem:central_action}. In particular, the spaces $\cal CW$ are uniformly hyperbolic.
    
    Given $[W]$, if $\cal C W$ is bounded then so is $\cal C [W]$, and projections can be defined arbitrarily. Otherwise there is a single representative $W\in [W]$, and given $ [x]\in \cal Y$ let $$\pi_{[W]}([x])=q_W\left(\bigcup_{x\in [x]} \pi_W(x)\right).$$ Similarly, for every $[V]\in \ofS$ which is properly nested in, or transverse to, $[W]$, set 
    $$\rho^{[V]}_{[W]}=q_W\left(\bigcup_{V\in [V]} \rho^V_W\right).$$ 
    We now check that $(\cal Y, \ofS)$ is a HHG structure for $E'/Z$.
    \begin{itemize}
        \item \textbf{Projections~\ref{axiom:projections}.} For every $[W]\in \ofS$, $\pi_{[W]}$ is a $\delta$-coarsely onto, $\delta$-coarse map, as the quotient projection $q_W\colon \C W\to \C[W]$ is $1$-Lipschitz and surjective. Moreover, given $[x],[y]\in \cal Y$ let $x\in [x]$ and $y\in [y]$ realise the distance between $[x]$ and $[y]$, so that \[\dist_{[W]}([x],[y])\le \dist_{W}(x,y)\le \delta\dist_{\cal X}(x,y)+\delta=\delta\dist_{\cal Y}([x],[y])+\delta.\]
        
         \item \textbf{Nesting~\ref{axiom:nesting}.} By construction, the unique maximal element of $\ofS$ is $[S]$. Moreover, whenever $[V]\propnest [W]$ and $\cal C W$ is unbounded, we have that $\diam\rho^{[V]}_{[W]}\le \diam \bigcup_{k\in Z}\rho^{kV}_{W}$, and the latter is uniformly bounded since $Z$-orbits in $\cal C W$ are uniformly bounded.
        
        \item \textbf{Finite complexity~\ref{axiom:finite_complexity}.} By how nesting is defined, every chain $[U_1]\propnest\ldots\propnest [U_k]$ in $\ofS$ lifts to a chain $U_1\propnest \ldots\propnest U_k$ in $\wfS\subseteq \frakS''$, and the finite complexity axiom for $(\cal X, \frakS'')$ uniformly bounds the length of the latter chain.
        
        \item \textbf{Orthogonality~\ref{axiom:orthogonal}.} We first prove that $\nest$ and $\orth$ are pairwise exclusive. Indeed, suppose that $[V],[W]\in \ofS$ admit representatives $V,V'\in [V]$ and $W,W'\in [W]$ such that $V\orth W$ and $W'\propnest V'$. Up to the $Z$-action we can assume that $V=V'$. But then, since $W\in \wfS\subseteq \frakS''$, there exists $U\nest W$ with $\cal C U$ unbounded, which is fixed by $Z$ and is therefore nested in $W'$. This contradicts the fact that, in $\frakS''$, orthogonal domains have no common nested domain.
        \\
        Regarding the second requirement, let $[U],[V],[W]\in \ofS$ be such that $[U]\propnest [V]$ and $[V]\orth [W]$. As above, we can find representatives $U\propnest V$ and $V\orth W$, so we get that $U\orth W$ and therefore $[U]\orth [W]$.

        \item \textbf{Containers~\ref{axiom:containers}.} Let $[V]\propnest [W]$, and let $[V']\nest [V]$ be such that $\cal C V'$ is unbounded. Let $\cal Q=\{[Q]\propnest [W]\mid [Q]\orth[V]\}$, and notice that every element of $\cal Q$ is orthogonal to $[V']$ as well. Fix a representative $W\in [W]$, and for every $[Q]\in \cal Q$ let $Q\in [Q]$ be nested in $W$. Since $V'$ is the unique representative of $[V']$, it must be nested in $W$ and orthogonal to every $Q$. Thus the container axiom for $\frakS''$ produces a domain $T\propnest W$ which contains every $Q$. Hence $[T]\propnest [W]$ and contain all $[Q]\in \cal Q$, so it is a container for $[V]$ inside $[W]$.

        \item \textbf{Transversality~\ref{axiom:transversality}.} Uniform boundedness of projection points follows as for the nesting axioms.
        
        \item \textbf{Consistency, bounded geodesic image, and large links~\ref{axiom:consistency}-\ref{axiom:large_link_lemma}} Since $q_W$ is a uniform quasi-isometry for every $W\in \ofS$, all three axioms hold as they were true in $(\cal X, \frakS'')$, plus the fact that we defined (relative) projections via the original (relative) projections.
        
        \item \textbf{Partial realisation~\ref{axiom:partial_realisation}.} Let $[V_1],\ldots,[V_k]\in \ofS$ be pairwise orthogonal, and for every $i$ let $p_i\in \C [V_i]$. For every $i$ let $r_i\in q_{V_i}^{-1}(p_i)$, and let $x\in \cal X$ realise the collection $\{r_i\}_{i=1,\ldots, k}$. It is easy to see that $[x]$ realises $\{p_i\}_{i=1,\ldots, k}$ in $\cal Y$.
        
        \item \textbf{Uniqueness~\ref{axiom:uniqueness}.} Let $[x],[y]\in\cal Y$, and let $r>0$ be a constant such that $\dist_{[W]}([x],[y])\le r$ for all $[W]\in\ofS$. Then any $x\in [x]$ and $y\in [y]$ have uniformly close projections to all $W\in \wfS$ (again, because every $q_W$ is a uniform quasi-isometry). Furthermore, since $Z$ acts coboundedly on $\C U_z$, we can replace $y$ by a $K$-translate and uniformly bound $\dist_{U_z}(x,y)$ as well; notice that this replacement preserves the distance between $x$ and $y$ to every $W\in \wfS$ up to a bounded error, since $Z$-orbits have uniformly bounded projections to $\C W$. Then the uniqueness axiom for $(\cal X, \frakS'')$ yields that $x$ and $y$ are uniformly close in $\cal X$, and therefore $\dist_{\cal Y}([x],[y])\le \dist_{\cal X}(x,y)$ is uniformly bounded.        
        \item \textbf{HHG structure.} Since $E'$ acted metrically properly and coboundedly on $\cal X$, then so does $E'/Z$ on $\cal Y$. Moreover, the cofinite $E'$-action on $\wfS$ induces a cofinite $E'/Z$-action on $\ofS$, so $(\cal Y, \ofS)$ is a HHG structure for $E'/Z$.
    \end{itemize}

    \par\medskip\textbf{Step 5: conclusion} As argued in the Interlude above, to conclude the proof of Theorem~\ref{thm:centralquot_is_HHS} we are left to show that the collection $\{[z_2],\ldots, [z_n]\}\subset E'/K$ is clean in $\ofS$. By how orthogonality and nesting are defined in $\ofS$, it is in turn enough to prove that $\Bigdom{[z_i]}\subseteq [\cal U]$, where $[\cal U]=\{[U]\}_{U\in \cal U-\{U_z\}}$. Indeed, if $[W]\in \ofS$ is unbounded, and $\diam_{[W]}(\langle [z_i]\rangle \cdot [x_0])=\infty$ for some $x_0\in \cal X$, then $\diam_{W}(\langle z_i\rangle \cdot x_0)=\infty$ as well, since the projection $q_W\colon \C W\to \C [W]$ is a quasi-isometry. Hence $W\in \Bigdom{z_i}-\{U_z\}$, so $[W]\in [\cal U]$, as required.
\end{proof}

\begin{rem}
\label{re-mark}
    In the statement of Theorem~\ref{thm:centralquot_is_HHS}, one cannot hope that $G$ is genuinely a HHG. Indeed, the $(3,3,3)$ triangle group $G$ is not a HHG \cite[Corollary 4.5]{PS_unbounded}; however the direct product $G\times \Z$ is a HHG with clean containers, as it acts geometrically on the standard cubulation of $\R^3$. We provide two arguments.
    
    The first one, more conceptual, was suggested to us by an anonymous referee. The equilateral triangle tiling of the plane, on which $G$ acts geometrically, is exactly the $A_2$ root lattice, which can be constructed as the set of integer vertices in the plane orthogonal to $(1,1,1)$ in $\R^3$. The translation along $(1,1,1)$ commutes with the natural action of $G$, giving an action of $G\times \Z$, which can be shown to be geometric.
      
    The second argument is more explicit and was suggested to us by Mark Hagen.  Let $o = (0,0,0)$ and $p = (1,1,1)$. Let $H_1$ be the plane $\{x=y\}$, and similarly define $H_2=\{y=z\}$ and $H_3=\{x=z\}$. Let $T_1, T_2, T_3$ be the reflections across $H_1, H_2, H_3$, which preserve the standard cubulation of $\R^3$. Now let $F=\{x+y+z=0\}$, which is the plane orthogonal to $op$ passing through $o$, and let $L(x,y,z)=(x+1,y+1,z-2)$, which is an integer translation along the line in which $H_1$ intersects $F$. Notice that $T_1,T_2,T_3$ preserve $F$ as they all fix $op$, and moreover so does $L$. Then let $G = \langle \alpha,\beta,\gamma\rangle$, where
    \begin{itemize}
    \item $\alpha=T_1\colon (x,y,z)\mapsto (y,x,z)$,
    \item $\beta=T_2\colon (x,y,z)\mapsto(x,z,y)$,
    \item $\gamma=LT_3L^{-1}\colon(x,y,z)\mapsto (z-3,y,x+3)$.
\end{itemize}
$G$ is a quotient of the $(3,3,3)$ triangle group $H$, since the above elements are all reflections and the product of any two of them has order $3$ (geometrically, this is because the dihedral angle between any two of the planes is $\pi/3$). Also $\alpha\circ\beta\circ\gamma$ is given by $(x, y, z) \mapsto (x+3, z-3, y)$, which has infinite order (as one sees by looking at the first coordinate). However, $H$ is \emph{just-infinite}, meaning that whenever $H$ surjects onto an infinite group $G$ then $H\cong G$. Indeed, this follows from \cite[Proposition 9]{McCarthy_just_infinite}, together with the fact that the point group is $S_3$, realised as an irreducible subgroup of $\mathrm{GL}_2(\mathbb{Z})$. This shows that $G$ is isomorphic to the $(3,3,3)$ triangle group.

Finally, let $S(x,y,z)=(x+1,y+1,z+1)$, which commutes with $L$ and with every $T_i$ as $op\subset H_1\cap H_2\cap H_3$. Moreover $\langle S\rangle \cap G$ is trivial, since $G$ fixes $F$. Hence $E=\langle S,G\rangle\cong \langle S\rangle \times G$ is a direct product. Notice that $E$ preserves the standard cubulation of $\R^3$, as its generators do, and the action is proper and cocompact (as one can see by considering the plane $F$ and its translates by powers of $S$). By e.g. \cite[Remark 13.2]{HHS_I}, $E$ is a HHG, whose domain set consists of:
\begin{itemize}
    \item the top element $S$;
    \item three elements for the coordinate planes, which are pairwise transverse;
    \item three elements for the coordinate lines, each of which is nested in the two planes it belongs and orthogonal to the third one.
\end{itemize} 
In particular this structure has clean containers, as required.
\end{rem}

\section{Quotients of mapping class groups} \label{sec:quot_MCG}
\subsection{The general problem}
We now discuss certain quotient extensions that it would be interesting to understand. The context is quotients of mapping class groups, and in particular the conjecture of the second and third author \cite[Question 3]{MS_rigidity_MCG}, which we now state in a simplified form.

Let $\calG(S)$ be the mapping class group of a finite-type surface $S$, and let $\cal{B}=\{\phi_1,\dots,\phi_k\}$ be a collection of elements. The conjecture predicts that there exists $M\in \mathbb{N}_{>0}$ such that $\calG(S)/\N$ is hierarchically hyperbolic, where $$\N=\ll \phi_1^{M},\ldots, \phi_k^M \rr.
$$

There are various cases known in the literature. For instance, \cite[Theorem 7.1]{BHMS} provides an affirmative answer if $\cal B$ consists of conjugacy representatives of all Dehn twists; \cite{short_HHG:II} almost completely settles the conjecture for the five-punctured sphere; and \cite{random_quot} proves hierarchical hyperbolicity of quotients by random walks, which should be thought as the ``typical'' quotients (see also \cite{hhs_asdim} for similar classes of quotients, including some of those considered in \cite{claymangahas}). The simplest unknown case is where $\cal B$ consists of a single Dehn twist, and even in that case one already encounters the problem of showing that a certain quotient central extension is bounded.
 
More precisely, if $p\colon \calG(S)\to E$ is the quotient under consideration, let $H\le E$ be the image under $p$ of the stabiliser of a curve $\gamma$, with the property that no power of the Dehn twist around $\gamma$ is conjugate into an element of $\cal B$. In any reasonable HHG structure, $H$ should be itself hierarchically hyperbolic, as it would correspond to a \emph{standard product region} (see e.g. \cite[Definition 5.15]{HHS_II}). One of the simplest scenarios is where $\gamma$ is non-separating, so that its stabiliser is the mapping class group of $S-\gamma$. Hence, a simplified, yet significant version of the problem is the following. 

\begin{question}\label{quest:mcg/dt_interior}
Let $U$ be a finite-type surface with one boundary component $\gamma$, let $\h U$ be the surface obtained by gluing a once-punctured disk to $\gamma$ with puncture $p$, and let $\calG(\h U,p)$ be the subgroup of the mapping class group of $\h U$ fixing $p$. By e.g. \cite[Proposition 3.19]{FarbMargalit}, there is a short exact sequence
   \[
       1\to \langle \tau_\gamma\rangle \to\calG(U)\to\calG(\h U,p)\to1.
    \]
   Let $\N_\gamma\le \calG(U)$ be the normal subgroup generated by all $M$-th powers of Dehn twists along curves which are supported on the interior of $U$.  
Is $\calG(U)/\N_\gamma$ hierarchically hyperbolic, for a suitable choice of $M$?
\end{question} 

It can be deduced from \cite{dahmani:rotating} that $\N_\gamma$ intersects the kernel trivially, so the quotient is itself a central extension 
\[
    1\to \langle \tau_\gamma\rangle \to\calG(U)/\N_\gamma\to\calG(\h U, p)/\pi(\N_\gamma)\to1.
\]
Furthermore, the aforementioned \cite[Theorem 7.1]{BHMS} yields that the base $\calG(\h U, p)/\pi(\N_\gamma)$ is itself hierarchically hyperbolic, for a suitable choice of $M$; hence Question~\ref{quest:mcg/dt_interior} is equivalent to the boundedness of the second central extension, by Theorem \ref{thm:HHG_iff_bounded_abelian_ker}.

\subsection{Quotients by pseudo-Anosovs}
Another interesting (and possibly easier) version of Question \ref{quest:mcg/dt_interior} is to ask about quotients by powers of pseudo-Anosovs rather than Dehn twists.

\begin{question}
\label{quest:quot_by_pA}
Let $U$ be a finite-type surface with one boundary component $\gamma$, and let $\h U$ be the surface obtained by gluing a once-punctured disk to $\gamma$ with puncture $p$.  Let $h$ be a pseudo-Anosov mapping class, and let $\N_h$ be the normal closure of a sufficiently high power of $h$. Is the central extension
\[
    1\to \langle \tau_\gamma\rangle \to\calG(U)/\N_h\to\calG(\h U,p)/\pi(\N_h)\to1
\]
    hierarchically hyperbolic? Equivalently, since the base is hierarchically hyperbolic by \cite[Theorem 6.2]{hhs_asdim}, is the extension bounded?
\end{question}

The first version of this paper answered Question~\ref{quest:quot_by_pA} affirmatively in the case where $U$ is a 4-punctured disk, whose mapping class group is the braid group $B_4$ on four strands. However, since then, a complete answer has been obtained by Tao \cite[Corollary 1.7]{tao2025extensiontheoremquasimorphisms}. Since our proof was quite technical, instead of giving the details we shall highlight its overall strategy and some key ideas, since those are different from Tao's and might apply in other settings.

\begin{thm}
\label{thm:braids}
Let $h$ be a pseudo-Anosov element of $B_4$, and let $\hat h$ be its image in $\calG(S_5, p)$. Then there exists $M_0\in\mathbb{N}_{>0}$ such that for all multiples $M$ of $M_0$ the following holds. Let $\N_h\unlhd B_4$ (resp. $\N\unlhd \calG(S_5, p)$) be the normal subgroup generated by $h^M$ (resp. $\hat{h}^M$), with quotient map $q$ (resp. $\h q$). Then there is a commutative diagram
$$\begin{tikzcd}
    1\ar{r}& \langle \tau_\gamma\rangle \ar{r}\ar[d,"\cong"]&B_4\ar{r}\ar[d,"q"]&\calG(S_5, p)\ar{r}\ar[d,"\h q"]&1.\\
    1\ar{r}& \mathbb Z\ar{r}&B_4/\N_h\ar{r}&\calG(S_5, p)/\N\ar{r}&1.
\end{tikzcd}$$
    where the bottom row is a bounded central extension of a HHG. In particular, $B_4/\N_h$ is hierarchically hyperbolic.
\end{thm}

\begin{proof}[Proof idea]
Firstly, one proves that $\N_h$ intersects $\langle \tau_\gamma\rangle$ trivially, justifying that there is indeed a diagram as in the statement whose left vertical arrow is an isomorphism. This can be done using the machinery of \emph{rotating families}, introduced in \cite{DGO} to model the normal closure of a pseudo-Anosov (among many other examples). 

Now, in order to show boundedness of the relevant central extension, it suffices to exhibit a further quotient central extension
$$\begin{tikzcd}
    1\ar{r}& \mathbb Z\ar{r}\ar[d,"\cong"]&B_4/\N_h\ar{r}\ar{d}&\calG(S_5, p)/\N\ar{r}\ar{d}&1.\\
    1\ar{r}& \mathbb{Z} \ar{r}&H \ar{r}&\hat H\ar{r}&1,
\end{tikzcd}$$
where $\hat H$ is hyperbolic. Indeed, the comparison map for hyperbolic groups is surjective \cite{neumannreeves} (see also \cite{mineyev}); hence, if the bottom extension exists, it is bounded, and in turn the central extension from the statement is also bounded by Lemma \ref{qce pullback}. 

One can argue that $H$ can be taken to be the quotient of $G=B_4/\N_h$ by the normal subgroup $\cal K$ generated by suitable powers of all images of Dehn twists of $B_4$. In order to prove this, one needs techniques from \cite{dahmani:rotating} to show that the intersection $\cal K\cap \Z$ is trivial, and then the machinery from \cite{short_HHG:I,short_HHG:II}, in particular \cite[Theorem 4.1]{short_HHG:II}, to prove that $G/\Z\cal K$ is hyperbolic. Note that we can apply \cite{short_HHG:I,short_HHG:II} because $\calG(S_5, p)$ is ``sufficiently low-complexity'' to be a so-called \emph{short HHG}, as introduced in \cite{short_HHG:I}.
\end{proof}

\begin{rem}\label{rem:enough_hyp_rf}
    Under the assumption that ``enough'' hyperbolic groups are residually finite, \cite{BHMS} constructs hyperbolic quotients of mapping class groups. One could in principle try to exploit these quotients, or variations, and a similar strategy as in the proof above to solve some cases of Question \ref{quest:mcg/dt_interior}, possibly conditionally on the residual finiteness of the relevant hyperbolic groups. The main difficulty is ensuring that kernels to these hyperbolic quotients, which morally arise as a union of kernels by rotating families, lift isomorphically to the corresponding central extension. 
\end{rem}

\bibliography{biblio}

@article {BHMS,
    AUTHOR = {Behrstock, Jason and Hagen, Mark and Martin, Alexandre and
              Sisto, Alessandro},
     TITLE = {A combinatorial take on hierarchical hyperbolicity and
              applications to quotients of mapping class groups},
   JOURNAL = {J. Topol.},
  FJOURNAL = {Journal of Topology},
    VOLUME = {17},
      YEAR = {2024},
    NUMBER = {3},
     PAGES = {Paper No. e12351, 94},
      ISSN = {1753-8416,1753-8424},
   MRCLASS = {20 (57K20 57M07)},
  MRNUMBER = {4822919},
       DOI = {10.1112/topo.12351},
       URL = {https://doi.org/10.1112/topo.12351},
}

@unpublished{tao2025properactionsfiniteproducts,
      title={Proper actions on finite products of hyperbolic spaces}, 
      author={Bingxue Tao and Renxing Wan},
      note={arXiv preprint arXiv:2506.04856},
      year={2025},
}

@unpublished{tao2025extensiontheoremquasimorphisms,
      title={An extension theorem for quasimorphisms}, 
      author={Bingxue Tao},
      note={arXiv preprint arXiv:2511.21306},
      year={2025},
}

@article {dahmani:rotating,
    AUTHOR = {Dahmani, Fran\c{c}ois},
     TITLE = {The normal closure of big {D}ehn twists and plate spinning
              with rotating families},
   JOURNAL = {Geom. Topol.},
  FJOURNAL = {Geometry \& Topology},
    VOLUME = {22},
      YEAR = {2018},
    NUMBER = {7},
     PAGES = {4113--4144},
      ISSN = {1465-3060},
   MRCLASS = {20F65 (20E07)},
  MRNUMBER = {3890772},
MRREVIEWER = {Michael Hull},
       DOI = {10.2140/gt.2018.22.4113},
       URL = {https://doi.org/10.2140/gt.2018.22.4113},
}

@article {neumannreeves,
    AUTHOR = {Neumann, Walter D. and Reeves, Lawrence},
     TITLE = {Central extensions of word hyperbolic groups},
   JOURNAL = {Ann. of Math. (2)},
  FJOURNAL = {Annals of Mathematics. Second Series},
    VOLUME = {145},
      YEAR = {1997},
    NUMBER = {1},
     PAGES = {183--192},
      ISSN = {0003-486X},
   MRCLASS = {20F32 (20J05)},
  MRNUMBER = {1432040},
MRREVIEWER = {John Meier},
       DOI = {10.2307/2951827},
       URL = {https://doi.org/10.2307/2951827},
}

@article{HRSS_3manifold,
    AUTHOR = {Hagen, Mark and Russell, Jacob and Sisto, Alessandro and
              Spriano, Davide},
     TITLE = {Equivariant hierarchically hyperbolic structures for
              3-manifold groups via quasimorphisms},
   JOURNAL = {Ann. Inst. Fourier (Grenoble)},
  FJOURNAL = {Universit\'e{} de Grenoble. Annales de l'Institut Fourier},
    VOLUME = {75},
      YEAR = {2025},
    NUMBER = {2},
     PAGES = {769--828},
      ISSN = {0373-0956,1777-5310},
   MRCLASS = {20F67 (57K35)},
  MRNUMBER = {4921365},
       DOI = {10.5802/aif.3654},
       URL = {https://doi.org/10.5802/aif.3654},
}

@article {DHS,
    AUTHOR = {Durham, Matthew Gentry and Hagen, Mark F. and Sisto,
              Alessandro},
     TITLE = {Boundaries and automorphisms of hierarchically hyperbolic
              spaces},
   JOURNAL = {Geom. Topol.},
  FJOURNAL = {Geometry \& Topology},
    VOLUME = {21},
      YEAR = {2017},
    NUMBER = {6},
     PAGES = {3659--3758},
      ISSN = {1465-3060},
   MRCLASS = {20F65 (20F67 30F60)},
  MRNUMBER = {3693574},
MRREVIEWER = {Nadia Benakli},
       DOI = {10.2140/gt.2017.21.3659},
       URL = {https://doi.org/10.2140/gt.2017.21.3659},
}

@article {HHS_I,
    AUTHOR = {Behrstock, Jason and Hagen, Mark F. and Sisto, Alessandro},
     TITLE = {Hierarchically hyperbolic spaces, {I}: {C}urve complexes for
              cubical groups},
   JOURNAL = {Geom. Topol.},
  FJOURNAL = {Geometry \& Topology},
    VOLUME = {21},
      YEAR = {2017},
    NUMBER = {3},
     PAGES = {1731--1804},
      ISSN = {1465-3060},
   MRCLASS = {20F36 (20F55 20F65)},
  MRNUMBER = {3650081},
MRREVIEWER = {Nadia Benakli},
       DOI = {10.2140/gt.2017.21.1731},
       URL = {https://doi.org/10.2140/gt.2017.21.1731},
}

@article {HHS_II,
    AUTHOR = {Behrstock, Jason and Hagen, Mark and Sisto, Alessandro},
     TITLE = {Hierarchically hyperbolic spaces {II}: {C}ombination theorems
              and the distance formula},
   JOURNAL = {Pacific J. Math.},
  FJOURNAL = {Pacific Journal of Mathematics},
    VOLUME = {299},
      YEAR = {2019},
    NUMBER = {2},
     PAGES = {257--338},
      ISSN = {0030-8730,1945-5844},
   MRCLASS = {20F36 (20F65 20F67)},
  MRNUMBER = {3956144},
MRREVIEWER = {Jiming\ Ma},
       DOI = {10.2140/pjm.2019.299.257},
       URL = {https://doi.org/10.2140/pjm.2019.299.257},
}

@article{masurminsky1,
 author = {Masur, Howard A. and Minsky, Yair N.},
 title = {Geometry of the complex of curves. {I}: {Hyperbolicity}},
 fjournal = {Inventiones Mathematicae},
 journal = {Invent. Math.},
 issn = {0020-9910},
 volume = {138},
 number = {1},
 pages = {103--149},
 year = {1999},
 language = {English},
 doi = {10.1007/s002220050343},
 zbMATH = {1355494},
 Zbl = {0941.32012}
}

@article{McCarthy_just_infinite,
author = {McCarthy, Donald},
title = {Infinite groups whose proper quotient groups are finite, I},
journal = {Communications on Pure and Applied Mathematics},
volume = {21},
number = {6},
pages = {545-562},
doi = {https://doi.org/10.1002/cpa.3160210604},
url = {https://onlinelibrary.wiley.com/doi/abs/10.1002/cpa.3160210604},
eprint = {https://onlinelibrary.wiley.com/doi/pdf/10.1002/cpa.3160210604},
year = {1968}
}

@article{masurminsky2,
 author = {Masur, H. A. and Minsky, Y. N.},
 title = {Geometry of the complex of curves. {II}: {Hierarchical} structure},
 fjournal = {Geometric and Functional Analysis. GAFA},
 journal = {Geom. Funct. Anal.},
 issn = {1016-443X},
 volume = {10},
 number = {4},
 pages = {902--974},
 year = {2000},
 language = {English},
 doi = {10.1007/PL00001643},
 zbMATH = {1545126},
 Zbl = {0972.32011}
}

@article {ABO,
    AUTHOR = {Abbott, Carolyn and Balasubramanya, Sahana H. and Osin, Denis},
     TITLE = {Hyperbolic structures on groups},
   JOURNAL = {Algebr. Geom. Topol.},
  FJOURNAL = {Algebraic \& Geometric Topology},
    VOLUME = {19},
      YEAR = {2019},
    NUMBER = {4},
     PAGES = {1747--1835},
      ISSN = {1472-2747,1472-2739},
   MRCLASS = {20F65 (20E08 20F67)},
  MRNUMBER = {3995018},
MRREVIEWER = {Tushar\ Das},
       DOI = {10.2140/agt.2019.19.1747},
       URL = {https://doi.org/10.2140/agt.2019.19.1747},
}

@article {Manning_actions_on_hyp,
    AUTHOR = {Manning, Jason Fox},
     TITLE = {Actions of certain arithmetic groups on {G}romov hyperbolic
              spaces},
   JOURNAL = {Algebr. Geom. Topol.},
  FJOURNAL = {Algebraic \& Geometric Topology},
    VOLUME = {8},
      YEAR = {2008},
    NUMBER = {3},
     PAGES = {1371--1402},
      ISSN = {1472-2747,1472-2739},
   MRCLASS = {53C24 (20E08 20F65 22E40)},
  MRNUMBER = {2443247},
MRREVIEWER = {Lizhen\ Ji},
       DOI = {10.2140/agt.2008.8.1371},
       URL = {https://doi.org/10.2140/agt.2008.8.1371},
}

@unpublished{short_HHG:I,
title={Short hierarchically hyperbolic groups {I}: uncountably many coarse median structures},
  author={Mangioni, Giorgio},
  note={arXiv preprint arXiv:2410.09232},
  year={2024}
}

@unpublished {random_quot,
      title={Random quotients preserve acylindrical and hierarchical hyperbolicity}, 
      author={Carolyn Abbott and Daniel Berlyne and Giorgio Mangioni and Thomas Ng and Alexander J. Rasmussen},
      note={arXiv preprint arXiv:2507.16677},
      year={2025}
}

@article{short_HHG:II,
title = {Short hierarchically hyperbolic groups {II}: Quotients and the {Hopf} property for {Artin} groups},
journal = {Advances in Mathematics},
volume = {486},
pages = {110736},
year = {2026},
issn = {0001-8708},
doi = {https://doi.org/10.1016/j.aim.2025.110736},
url = {https://www.sciencedirect.com/science/article/pii/S0001870825006346},
author = {Mangioni, G. and Sisto, A.},
}

@article{MS_rigidity_MCG,
    author = {Mangioni, Giorgio and Sisto, Alessandro},
    title = {Rigidity of mapping class groups mod powers of twists},
    journal = {Proceedings of the Royal Society of Edinburgh: Section A Mathematics},
    year = {2025},
    DOI ={10.1017/prm.2025.23},
pages = {1-71}
}

@book{calegari,
    AUTHOR = {Calegari, Danny},
     TITLE = {scl},
    SERIES = {MSJ Memoirs},
    VOLUME = {20},
 PUBLISHER = {Mathematical Society of Japan, Tokyo},
      YEAR = {2009},
     PAGES = {xii+209},
      ISBN = {978-4-931469-53-2},
   MRCLASS = {57M07 (20F65 20F67 20J05 37E45)},
  MRNUMBER = {2527432},
MRREVIEWER = {Athanase\ Papadopoulos},
       DOI = {10.1142/e018},
       URL = {https://doi.org/10.1142/e018},
}

@article{extendable:survey,
    AUTHOR = {Kawasaki, Morimichi and Kimura, Mitsuaki and Maruyama, Shuhei
              and Matsushita, Takahiro and Mimura, Masato},
     TITLE = {Survey on invariant quasimorphisms and stable mixed commutator
              length},
   JOURNAL = {Topology Proc.},
  FJOURNAL = {Topology Proceedings},
    VOLUME = {64},
      YEAR = {2024},
     PAGES = {129--174},
      ISSN = {0146-4124},
   MRCLASS = {20F65 (20E36 20F12 20J06 57M07)},
  MRNUMBER = {4773366},
MRREVIEWER = {Francesco G. Russo},
}

@unpublished{Durham_cubinf,
 author = {Durham, Matthew Gentry},
 title = {Cubulating {Infinity} in {Hierarchically} {Hyperbolic} {Spaces}},
 note = {arXiv preprint arXiv:2308.13689},
 year = {2023},
}

@article {extendable:bavard,
    AUTHOR = {Kawasaki, Morimichi and Kimura, Mitsuaki and Matsushita,
              Takahiro and Mimura, Masato},
     TITLE = {Bavard's duality theorem for mixed commutator length},
   JOURNAL = {Enseign. Math.},
  FJOURNAL = {L'Enseignement Math\'{e}matique},
    VOLUME = {68},
      YEAR = {2022},
    NUMBER = {3-4},
     PAGES = {441--481},
      ISSN = {0013-8584},
   MRCLASS = {20F12 (05E45 20F36 20J06 46B10)},
  MRNUMBER = {4452430},
MRREVIEWER = {Valeriy G. Bardakov},
       DOI = {10.4171/lem/1037},
       URL = {https://doi.org/10.4171/lem/1037},
}

@article{extendable:main,
  title={The space of non-extendable quasimorphisms},
  author={Kawasaki, Morimichi and Kimura, Mitsuaki and Maruyama, Shuhei and Matsushita, Takahiro and Mimura, Masato},
journal={Alg. Geom. Top.},
year={2025},
volume={25},
pages={1169–1226},
doi={10.2140/agt.2025.25.1169},
}

@article {FPS,
    AUTHOR = {Frigerio, R. and Pozzetti, M. B. and Sisto, A.},
     TITLE = {Extending higher-dimensional quasi-cocycles},
   JOURNAL = {J. Topol.},
  FJOURNAL = {Journal of Topology},
    VOLUME = {8},
      YEAR = {2015},
    NUMBER = {4},
     PAGES = {1123--1155},
      ISSN = {1753-8416},
   MRCLASS = {20J06 (20F65 57M07 57N16)},
  MRNUMBER = {3431671},
       DOI = {10.1112/jtopol/jtv017},
       URL = {https://doi.org/10.1112/jtopol/jtv017},
}

@article{frigsisto,
 author = {Frigerio, Roberto and Sisto, Alessandro},
 title = {Central extensions and bounded cohomology},
 fjournal = {Annales Henri Lebesgue},
 journal = {Ann. Henri Lebesgue},
 issn = {2644-9463},
 volume = {6},
 pages = {225--258},
 year = {2023},
 language = {English},
 doi = {10.5802/ahl.164},
 zbMATH = {7734709},
 Zbl = {1522.20170}
}

@article {Hull_Osin,
    AUTHOR = {Hull, Michael and Osin, Denis},
     TITLE = {Induced quasicocycles on groups with hyperbolically embedded
              subgroups},
   JOURNAL = {Algebr. Geom. Topol.},
  FJOURNAL = {Algebraic \& Geometric Topology},
    VOLUME = {13},
      YEAR = {2013},
    NUMBER = {5},
     PAGES = {2635--2665},
      ISSN = {1472-2747,1472-2739},
   MRCLASS = {20F65 (20F67 20J06 43A15 57M07)},
  MRNUMBER = {3116299},
MRREVIEWER = {Alexander\ Fel\cprime shtyn},
       DOI = {10.2140/agt.2013.13.2635},
       URL = {https://doi.org/10.2140/agt.2013.13.2635},
}

@article {HHP:coarse,
    AUTHOR = {Haettel, Thomas and Hoda, Nima and Petyt, Harry},
     TITLE = {Coarse injectivity, hierarchical hyperbolicity and
              semihyperbolicity},
   JOURNAL = {Geom. Topol.},
  FJOURNAL = {Geometry \& Topology},
    VOLUME = {27},
      YEAR = {2023},
    NUMBER = {4},
     PAGES = {1587--1633},
      ISSN = {1465-3060,1364-0380},
   MRCLASS = {20F65 (20F67 51F30)},
  MRNUMBER = {4602421},
       DOI = {10.2140/gt.2023.27.1587},
       URL = {https://doi.org/10.2140/gt.2023.27.1587},
}

@article {hhs_asdim,
    AUTHOR = {Behrstock, Jason and Hagen, Mark F. and Sisto, Alessandro},
     TITLE = {Asymptotic dimension and small-cancellation for hierarchically
              hyperbolic spaces and groups},
   JOURNAL = {Proc. Lond. Math. Soc. (3)},
  FJOURNAL = {Proceedings of the London Mathematical Society. Third Series},
    VOLUME = {114},
      YEAR = {2017},
    NUMBER = {5},
     PAGES = {890--926},
      ISSN = {0024-6115,1460-244X},
   MRCLASS = {20F65 (20F67 57M07)},
  MRNUMBER = {3653249},
MRREVIEWER = {Mahan\ Mj},
       DOI = {10.1112/plms.12026},
       URL = {https://doi.org/10.1112/plms.12026},
}

@article {DGO,
    AUTHOR = {Dahmani, F. and Guirardel, V. and Osin, D.},
     TITLE = {Hyperbolically embedded subgroups and rotating families in
              groups acting on hyperbolic spaces},
   JOURNAL = {Mem. Amer. Math. Soc.},
  FJOURNAL = {Memoirs of the American Mathematical Society},
    VOLUME = {245},
      YEAR = {2017},
    NUMBER = {1156},
     PAGES = {v+152},
      ISSN = {0065-9266,1947-6221},
   MRCLASS = {20F65 (20F06 20F34 20F67 57M07)},
  MRNUMBER = {3589159},
MRREVIEWER = {Dominik\ Gruber},
       DOI = {10.1090/memo/1156},
       URL = {https://doi.org/10.1090/memo/1156},
}

@article {quasiflat,
    AUTHOR = {Behrstock, Jason and Hagen, Mark F. and Sisto, Alessandro},
     TITLE = {Quasiflats in hierarchically hyperbolic spaces},
   JOURNAL = {Duke Math. J.},
  FJOURNAL = {Duke Mathematical Journal},
    VOLUME = {170},
      YEAR = {2021},
    NUMBER = {5},
 PAGES = {909--996},
      ISSN = {0012-7094,1547-7398},
   MRCLASS = {20F67 (20F36 20F55 30F60 53C23)},
  MRNUMBER = {4255047},
MRREVIEWER = {Yasushi\ Yamashita},
       DOI = {10.1215/00127094-2020-0056},
       URL = {https://doi.org/10.1215/00127094-2020-0056},
}

@article{uniform_undistortion,
    AUTHOR = {Abbott, Carolyn and Hagen, Mark and Petyt, Harry and Zalloum,
              Abdul},
     TITLE = {Uniform undistortion from barycentres, and applications to
              hierarchically hyperbolic groups},
   JOURNAL = {Anal. Geom. Metr. Spaces},
  FJOURNAL = {Analysis and Geometry in Metric Spaces},
    VOLUME = {13},
      YEAR = {2025},
    NUMBER = {1},
     PAGES = {Paper No. 20250027},
      ISSN = {2299-3274},
   MRCLASS = {99-06},
  MRNUMBER = {4993373},
       DOI = {10.1515/agms-2025-0027},
       URL = {https://doi.org/10.1515/agms-2025-0027},
}

@incollection {gromov,
    AUTHOR = {Gromov, M.},
     TITLE = {Hyperbolic groups},
 BOOKTITLE = {Essays in group theory},
    SERIES = {Math. Sci. Res. Inst. Publ.},
    VOLUME = {8},
     PAGES = {75--263},
 PUBLISHER = {Springer, New York},
      YEAR = {1987},
      ISBN = {0-387-96618-8},
   MRCLASS = {20F32 (20F06 20F10 22E40 53C20 57R75 58F17)},
  MRNUMBER = {919829},
MRREVIEWER = {Christopher\ W.\ Stark},
       DOI = {10.1007/978-1-4613-9586-7\_3},
       URL = {https://doi.org/10.1007/978-1-4613-9586-7_3},
}

@article {claymangahas,
    AUTHOR = {Clay, Matt and Mangahas, Johanna},
     TITLE = {Hyperbolic quotients of projection complexes},
   JOURNAL = {Groups Geom. Dyn.},
  FJOURNAL = {Groups, Geometry, and Dynamics},
    VOLUME = {16},
      YEAR = {2022},
    NUMBER = {1},
     PAGES = {225--246},
      ISSN = {1661-7207,1661-7215},
   MRCLASS = {20F65 (57M07)},
  MRNUMBER = {4424969},
MRREVIEWER = {Andrzej\ Szczepa\'nski},
       DOI = {10.4171/ggd/646},
       URL = {https://doi.org/10.4171/ggd/646},
}

@article{ANS:UEG,
 author = {Abbott, Carolyn R. and Ng, Thomas and Spriano, Davide and Gupta, Radhika and Petyt, Harry},
 title = {Hierarchically hyperbolic groups and uniform exponential growth},
 fjournal = {Mathematische Zeitschrift},
 journal = {Math. Z.},
 issn = {0025-5874},
 volume = {306},
 number = {1},
 pages = {33},
 note = {Id/No 18},
 year = {2024},
 language = {English},
 doi = {10.1007/s00209-023-03411-6},
 zbMATH = {7785008},
 Zbl = {1536.20052}
}

@article{HHL:proximal,
 author = {Horbez, Camille and Huang, Jingyin and L{\'e}cureux, Jean},
 title = {Proper proximality in non-positive curvature},
 fjournal = {American Journal of Mathematics},
 journal = {Am. J. Math.},
 issn = {0002-9327},
 volume = {145},
 number = {5},
 pages = {1327--1364},
 year = {2023},
 language = {English},
 doi = {10.1353/ajm.2023.a907700},
 zbMATH = {7771198},
 Zbl = {1527.20066}
}

@article{berlairobbio,
 author = {Berlai, Federico and Robbio, Bruno},
 title = {A refined combination theorem for hierarchically hyperbolic groups},
 fjournal = {Groups, Geometry, and Dynamics},
 journal = {Groups Geom. Dyn.},
 issn = {1661-7207},
 volume = {14},
 number = {4},
 pages = {1127--1203},
 year = {2020},
 language = {English},
 doi = {10.4171/GGD/576},
 zbMATH = {7362555},
 Zbl = {1511.20170}
}

@article{PS_unbounded,
 author = {Petyt, Harry and Spriano, Davide},
 title = {Unbounded domains in hierarchically hyperbolic groups},
 fjournal = {Groups, Geometry, and Dynamics},
 journal = {Groups Geom. Dyn.},
 issn = {1661-7207},
 volume = {17},
 number = {2},
 pages = {479--500},
 year = {2023},
 language = {English},
 doi = {10.4171/GGD/706},
 zbMATH = {7687945},
 Zbl = {1515.20245}
}

@article{ELTAG,
 author = {Hagen, Mark and Martin, Alexandre and Sisto, Alessandro},
 title = {Extra-large type {Artin} groups are hierarchically hyperbolic},
 fjournal = {Mathematische Annalen},
 journal = {Math. Ann.},
 issn = {0025-5831},
 volume = {388},
 number = {1},
 pages = {867--938},
 year = {2024},
 language = {English},
 doi = {10.1007/s00208-022-02523-4},
 zbMATH = {7796285},
 Zbl = {1536.20049}
}

@article{hagensusse,
 author = {Hagen, Mark F. and Susse, Tim},
 title = {On hierarchical hyperbolicity of cubical groups},
 fjournal = {Israel Journal of Mathematics},
 journal = {Isr. J. Math.},
 issn = {0021-2172},
 volume = {236},
 number = {1},
 pages = {45--89},
 year = {2020},
 language = {English},
 doi = {10.1007/s11856-020-1967-2},
 zbMATH = {7202560},
 Zbl = {1453.20056}
}

@article{graphprod,
 author = {Berlyne, Daniel and Russell, Jacob},
 title = {Hierarchical hyperbolicity of graph products},
 fjournal = {Groups, Geometry, and Dynamics},
 journal = {Groups Geom. Dyn.},
 issn = {1661-7207},
 volume = {16},
 number = {2},
 pages = {523--580},
 year = {2022},
 language = {English},
 doi = {10.4171/GGD/652},
 zbMATH = {7624152},
 Zbl = {1521.20085}
}

@article{DMS:stable,
 author = {Durham, Matthew G. and Minsky, Yair N. and Sisto, Alessandro},
 title = {Stable cubulations, bicombings, and barycenters},
 fjournal = {Geometry \& Topology},
 journal = {Geom. Topol.},
 issn = {1465-3060},
 volume = {27},
 number = {6},
 pages = {2383--2478},
 year = {2023},
 language = {English},
 doi = {10.2140/gt.2023.27.2383},
 zbMATH = {7781516},
 Zbl = {1535.20208}
}

@article{Gersten_B_is_QIT,
 author = {Gersten, S. M.},
 title = {Bounded cocycles and combings of groups},
 fjournal = {International Journal of Algebra and Computation},
 journal = {Int. J. Algebra Comput.},
 issn = {0218-1967},
 volume = {2},
 number = {3},
 pages = {307--326},
 year = {1992},
 language = {English},
 doi = {10.1142/S0218196792000190},
 zbMATH = {126427},
 Zbl = {0762.57001}
}

@book {frigerio,
    AUTHOR = {Frigerio, Roberto},
     TITLE = {Bounded cohomology of discrete groups},
    SERIES = {Mathematical Surveys and Monographs},
    VOLUME = {227},
 PUBLISHER = {American Mathematical Society, Providence, RI},
      YEAR = {2017},
     PAGES = {xvi+193},
      ISBN = {978-1-4704-4146-3},
   MRCLASS = {57T10 (20J06 37C85 55N10 57M07 57N16)},
  MRNUMBER = {3726870},
MRREVIEWER = {Clara L\"{o}h},
       DOI = {10.1090/surv/227},
       URL = {https://doi.org/10.1090/surv/227},
}

@article {heuer,
    AUTHOR = {Heuer, Nicolaus},
     TITLE = {Low-dimensional bounded cohomology and extensions of groups},
   JOURNAL = {Math. Scand.},
  FJOURNAL = {Mathematica Scandinavica},
    VOLUME = {126},
      YEAR = {2020},
    NUMBER = {1},
     PAGES = {5--31},
      ISSN = {0025-5521},
   MRCLASS = {20J06},
  MRNUMBER = {4087570},
MRREVIEWER = {Brita E. A. Nucinkis},
       DOI = {10.7146/math.scand.a-114969},
       URL = {https://doi.org/10.7146/math.scand.a-114969},
}

@article {mineyev,
    AUTHOR = {Mineyev, I.},
     TITLE = {Straightening and bounded cohomology of hyperbolic groups},
   JOURNAL = {Geom. Funct. Anal.},
  FJOURNAL = {Geometric and Functional Analysis},
    VOLUME = {11},
      YEAR = {2001},
    NUMBER = {4},
     PAGES = {807--839},
      ISSN = {1016-443X},
   MRCLASS = {20F67 (20F65 20J06 57M07)},
  MRNUMBER = {1866802},
MRREVIEWER = {Eric M. Freden},
       DOI = {10.1007/PL00001686},
       URL = {https://doi.org/10.1007/PL00001686},
}

@article {bavard,
    AUTHOR = {Bavard, Christophe},
     TITLE = {Longueur stable des commutateurs},
   JOURNAL = {Enseign. Math. (2)},
  FJOURNAL = {L'Enseignement Math\'{e}matique. Revue Internationale. 2e S\'{e}rie},
    VOLUME = {37},
      YEAR = {1991},
    NUMBER = {1-2},
     PAGES = {109--150},
      ISSN = {0013-8584},
   MRCLASS = {20F12 (20J05 57M07)},
  MRNUMBER = {1115747},
MRREVIEWER = {Darryl McCullough},
}

@book {brown,
    AUTHOR = {Brown, Kenneth S.},
     TITLE = {Cohomology of groups},
    SERIES = {Graduate Texts in Mathematics},
    VOLUME = {87},
      NOTE = {Corrected reprint of the 1982 original},
 PUBLISHER = {Springer-Verlag, New York},
      YEAR = {1994},
     PAGES = {x+306},
      ISBN = {0-387-90688-6},
   MRCLASS = {20J05 (20-02)},
  MRNUMBER = {1324339},
}

@article {fujikapo,
    AUTHOR = {Fujiwara, Koji and Kapovich, Michael},
     TITLE = {On quasihomomorphisms with noncommutative targets},
   JOURNAL = {Geom. Funct. Anal.},
  FJOURNAL = {Geometric and Functional Analysis},
    VOLUME = {26},
      YEAR = {2016},
    NUMBER = {2},
     PAGES = {478--519},
      ISSN = {1016-443X},
   MRCLASS = {20F65 (20F67)},
  MRNUMBER = {3513878},
MRREVIEWER = {Xiangdong Xie},
       DOI = {10.1007/s00039-016-0364-9},
       URL = {https://doi.org/10.1007/s00039-016-0364-9},
}

@article {kevin,
    AUTHOR = {Li, Kevin},
     TITLE = {Bounded cohomology of classifying spaces for families of
              subgroups},
   JOURNAL = {Algebr. Geom. Topol.},
  FJOURNAL = {Algebraic \& Geometric Topology},
    VOLUME = {23},
      YEAR = {2023},
    NUMBER = {2},
     PAGES = {933--962},
      ISSN = {1472-2747},
   MRCLASS = {20J06 (20F67 43A07 55N35)},
  MRNUMBER = {4587321},
MRREVIEWER = {Dirk Sch\"{u}tz},
       DOI = {10.2140/agt.2023.23.933},
       URL = {https://doi.org/10.2140/agt.2023.23.933},
}

@book {FarbMargalit,
    AUTHOR = {Farb, Benson and Margalit, Dan},
     TITLE = {A primer on mapping class groups},
    SERIES = {Princeton Mathematical Series},
    VOLUME = {49},
 PUBLISHER = {Princeton University Press, Princeton, NJ},
      YEAR = {2012},
     PAGES = {xiv+472},
      ISBN = {978-0-691-14794-9},
   MRCLASS = {57M50 (20F36 20F65 57M07 57N05)},
  MRNUMBER = {2850125},
MRREVIEWER = {Stephen P. Humphries},
}

@unpublished{FarrellJones,
 author = {Durham, Matthew Gentry and Minsky, Yair and Sisto, Alessandro},
 title = {Asymptotically {CAT}(0) metrics, {Z}-structures, and the {Farrell}-{Jones} {Conjecture}},
 note = {arXiv preprint arXiv:2504.17048},
 year = {2025},
}

@article {ascarimilizia,
    AUTHOR = {Ascari, Dario and Milizia, Francesco},
     TITLE = {Weakly bounded cohomology classes and a counterexample to a
              conjecture of {G}romov},
   JOURNAL = {Geom. Funct. Anal.},
  FJOURNAL = {Geometric and Functional Analysis},
    VOLUME = {34},
      YEAR = {2024},
    NUMBER = {3},
     PAGES = {631--658},
      ISSN = {1016-443X},
   MRCLASS = {20J06},
  MRNUMBER = {4743508},
MRREVIEWER = {Ian J. Leary},
       DOI = {10.1007/s00039-024-00676-9},
       URL = {https://doi.org/10.1007/s00039-024-00676-9},
}
\bibliographystyle{alpha}
\end{document}